\documentclass[11pt]{article}
\usepackage[a4paper, margin=1in]{geometry}

\RequirePackage{amsthm,amsmath,amsfonts,amssymb}
\RequirePackage[numbers]{natbib}
\RequirePackage[colorlinks,citecolor=blue,urlcolor=blue]{hyperref}
\RequirePackage{graphicx}
\graphicspath{{figures/}}

\usepackage{blindtext}

\usepackage[labelformat=simple]{subcaption}

\usepackage{amsmath}
\usepackage{algorithm,float}
\usepackage{algorithmic}
\usepackage{enumitem}
\usepackage{booktabs}
\usepackage{multirow}

\usepackage[normalem]{ulem}

\usepackage{graphicx}
\usepackage{placeins}
\usepackage{siunitx}

\usepackage{verbatim}
\usepackage{empheq}
\usepackage{rotating}
\usepackage{epsfig}

\allowdisplaybreaks
\theoremstyle{plain}

\newtheorem{theorem}{Theorem}[section]
\newtheorem{proposition}[theorem]{Proposition}

\theoremstyle{definition}
\newtheorem{definition}[theorem]{Definition}

\theoremstyle{remark}

\newtheorem{remark}[theorem]{Remark}
\newtheorem{assumption}{Assumption}

\usepackage{adjustbox}

\begin{document}
\definecolor{dred}{rgb}{0.8,0.2,0.0}
\newcommand{\blue}[1]{\begin{color}{blue}#1\end{color}}
\newcommand{\magenta}[1]{\begin{color}{magenta}#1\end{color}}
\newcommand{\red}[1]{\begin{color}{red}#1\end{color}}
\newcommand{\green}[1]{\begin{color}{green}#1\end{color}}
\newcommand{\dred}[1]{\begin{color}{dred}#1\end{color}}
\title{{Multiple Regression} for Matrix and Vector Predictors: Models, Theory, Algorithms, and Beyond}

\author{
  Meixia Lin\footnote{Engineering Systems and Design, Singapore University of Technology and Design, Singapore. Email: \href{mailto:meixia_lin@sutd.edu.sg}{meixia\_lin@sutd.edu.sg}} \and
  Ziyang Zeng\footnote{Department of Mathematics, National University of Singapore, Singapore. Email: \href{mailto:ziyangzeng@u.nus.edu}{ziyangzeng@u.nus.edu}} \and
  Yangjing Zhang\footnote{Institute of Applied Mathematics, Academy of Mathematics and Systems Science, Chinese Academy of Sciences, Beijing, China. Email: \href{mailto:yangjing.zhang@amss.ac.cn}{yangjing.zhang@amss.ac.cn}}
}

\maketitle
\begin{abstract}
Matrix regression plays an important role in modern data analysis due to its ability to handle complex relationships involving both matrix and vector variables.
We propose a class of regularized regression models capable of predicting both matrix and vector variables, accommodating various regularization techniques tailored to the inherent structures of the data. We establish the consistency of our
estimator when penalizing the nuclear norm of the matrix variable and the $\ell_1$ norm of the vector variable. To tackle the general regularized regression model, we propose
a unified framework based on
an efficient preconditioned proximal point algorithm. Numerical experiments demonstrate the superior estimation and prediction accuracy of our proposed estimator, as well as
the efficiency of our algorithm compared to the state-of-the-art solvers.

\end{abstract}




\section{Introduction}\label{sec:intro}
In the modern big data era, {it is commonplace to encounter data sets comprising both matrix and vector variate observations in various fields such as neuroimaging, transportation, and longitudinal studies, where data often exhibits multi-dimensional structures.} For example, the recent COVID-19 pandemic, which has led to significant global morbidity and mortality, underscores the importance of developing predictive tools for managing infectious diseases. The COVID-19 Open Data \citep{COVID} offers a rich repository of daily time-series data on COVID-19 cases, deaths, recoveries, testing, vaccinations, {as well as} hospitalizations, spanning over 230 countries, 760 regions, and 12,000 lower-level administrative divisions. This data set provides a unique opportunity to develop predictive models that can forecast the spread of the virus. By considering the total number of confirmed cases in each region over a specific period as the response variable, researchers can naturally employ {the} time-series data to construct a matrix predictor. Simultaneously, other variables with minimal variation during the target period, such as certain health-related characteristics, can serve as vector-valued predictors.  In addition to the COVID-19 data, other examples involving both matrix and vector variables include the bike sharing data \citep{7959977}, and the electroencephalography (EEG) data applied in diverse areas like predicting alcoholism \citep{misc_eeg_database_121} and the emotion recognition \citep{Onton2009High}. To handle such data structures, \citet{zhou2014regularized} proposed a matrix regression model
\begin{equation}
y_i = \langle X_i,B\rangle + \langle z_i,\gamma\rangle + \varepsilon_i,\quad i = 1,\cdots, n,\label{model}
\end{equation}
{where $\{(y_i,X_i,z_i),\,i=1,\dots,n\}$ are the observed data. Here} $y_i\in \mathbb{R}$ is the response, $X_i\in \mathbb{R}^{m\times q}$ is the matrix variate, $z_i\in \mathbb{R}^p$ is the vector variate, $B\in \mathbb{R}^{m\times q}$, $\gamma\in \mathbb{R}^p$ are the unknown regression coefficients, {and $\varepsilon_i$'s are independent and identically distributed random noise, each following} a normal distribution with mean $0$ and standard deviation $\sigma$.

Given the formulation of model \eqref{model} and the high dimensionality encountered in various applications, it is often presumed that the coefficients of the models
exhibit some special structures such as sparsity and low-rankness.
For instance, data sets like the COVID-19 Open Data, typically contain multiple features with longitudinal variability and the features that are relatively constant.
{See also \cite{fang2021matrix,fang2022regularized,hung2013matrix}, which assume a rank-$1$ canonical polyadic decomposition structure for the matrix variate.}
To address this challenge, we propose a class of regularized regression models capable of predicting structured matrix and vector variables. 
Given samples $\{(y_i,X_i,z_i),\,i=1,\dots,n\}$, we estimate the unknown regression coefficients $B$ and $\gamma$ through solving the following regularized {optimization} problem
\begin{align}
\min_{B\in \mathbb{R}^{m\times q}, \gamma \in \mathbb{R}^p} \  \sum_{i=1}^n
\ell(\langle X_i,B\rangle + \langle z_i,\gamma\rangle,y_i) + \phi(B) + \psi(\gamma), \label{eq: porg}
\end{align}
{where} $\ell:\mathbb{R}\times \mathbb{R}\rightarrow \mathbb{R}$ is a loss function, commonly chosen as  the
squared loss function {$\ell(s_i,y_i) =  (y_i - s_i)^2$. Here} $\phi:\mathbb{R}^{m\times q}\rightarrow (-\infty,\infty]$ {and} $\psi:\mathbb{R}^{p}\rightarrow (-\infty,\infty]$ are given closed proper convex functions, which serve as {the} penalty functions (also known as regularizers) encouraging special structures in $B$ and $\gamma$. Regularization is important due to the potential high dimensionality and intricate structure of the matrix and vector {predictors}. {Commonly used} penalty functions of vectors include the lasso \citep{tibshirani1996regression}, the fused lasso \citep{fusedlasso}, the (sparse) group lasso \citep{yuan2006model,friedman2010note}, and many others. For the matrix variate, the true predictor {often exhibits a low-rank structure or can be well approximated by one.} As such, a typical choice of the penalty function is the nuclear norm \citep{ma2011fixed,Toh2009AnAP}. There are also other penalty functions of matrices such as multivariate (sparse) group lasso \citep{obozinski2011support,li2015multivariate} and matrix-type fused lasso \citep{li2019double}. 

In the literature, various regularized regression models based on \eqref{model} have {also} been {studied}. {Initially, \citet{zhou2014regularized} considered model \eqref{eq: p} with  $\psi(\cdot) = 0$ and $\phi(\cdot)$ being spectral regularization, such as $\phi(B) = \|B\|_*$. In other words, their focus was primarily on estimating the matrix covariate without applying penalties to the vector part. Similarly, \citet{elsener2017robust} considered model \eqref{eq: p} with matrix variate only by assuming $\gamma = 0$, and they used the absolute value loss
and the Huber loss, together with a nuclear norm penalty $\phi(B) = \|B\|_*$.}
Later, \citet{li2019double} introduced matrix regression with a double fused lasso penalized least absolute deviation (LAD). That is, $\ell(s_i,y_i)=|y_i-s_i|$, with matrix-type fused lasso penalty  $\phi(B)=\lambda_1{\|B\|_*}+\lambda_2\sum_{j=2}^m\|B_{j\cdot}-B_{(j-1)\cdot}\|_1$ and vector-type fused lasso penalty  $\psi(\gamma)=\lambda_3\|\gamma\|_1+\lambda_4\sum_{j=2}^p{|\gamma_{j}-\gamma_{j-1}|}$.
Subsequently, \citet{Li2020linearized,li2021double} proposed double fused lasso regularized linear
and logistic regression.
All these models serve as specific cases of  \eqref{eq: porg}. In addition, our model \eqref{eq: p} accommodates both vector and matrix coefficients simultaneously and simplifies to the regularized matrix-variate (or vector-variate) regression model \citep{cui2024variable,elsener2017robust,10.1214/20-AOS1980,tibshirani1996regression,fusedlasso} under appropriate settings, illustrating its versatility and adaptability to both simpler and more complex scenarios involving various loss and penalty functions. As an increasing number of problems align with this framework, it becomes imperative to develop efficient algorithms to solve it.

Let {${\rm vec}(\cdot)$} be the operator stacking the columns of a matrix on top of one another into a vector. Denote $y = (y_1,\dots,y_n)^{\intercal}\in \mathbb{R}^n$, $Z = (z_1, \dots, z_n)^{\intercal} \in \mathbb{R}^{n\times p}$, and $X = ( {\rm vec}(X_1), \dots, {\rm vec}(X_n))^{\intercal}$\\$ \in \mathbb{R}^{n\times mq}$. Problem \eqref{eq: porg} can be written in a compact and general form:
\begin{equation}
\min_{B\in \mathbb{R}^{m\times q}, \gamma \in \mathbb{R}^p} \ \Big\{
G(B,\gamma):=  \hat{\ell} (X {\rm vec}(B) + Z \gamma, y) + \phi(B) + \psi(\gamma)
\Big\},\label{eq: p}
\end{equation}
where $\hat{\ell} :\mathbb{R}^{n} \times \mathbb{R}^{n} \rightarrow \mathbb{R}$ is a loss function. In particular, problem \eqref{eq: p} reduces to \eqref{eq: porg} if we take $\hat{\ell}(s,y) =   \sum_{i=1}^n \ell(s_i,y_i)$, for $s,y\in \mathbb{R}^{n}$. More generally, problem \eqref{eq: p} can also incorporate the square root loss function $\hat{\ell}(s,y) = \sqrt{ {\sum_{i=1}^n (y_i-s_i)^2/n}}$, which is not covered in {the form of} problem \eqref{eq: porg}. Throughout the paper, we make the blanket assumption that {$\hat{\ell}(\cdot,y)$ is a convex function for any $y\in \mathbb{R}^n$.}

In current research, {the} state-of-the-art first-order methods, including those based on Nesterov algorithm and the alternating direction method of multipliers (ADMM), are prevalent in previously discussed matrix regression {models}, with various adaptations for different regularization strategies. Specifically, \citet{zhou2014regularized} implemented the Nesterov algorithm for matrix regression problems, focusing on the nuclear norm regularizer. Later,  \citet{li2019double} introduced a symmetric Gauss-Seidel based ADMM (sGS-ADMM) to tackle matrix regression problems with a double fused lasso penalized LAD. Subsequently, \citet{li2021double} extended this approach to address double fused lasso regularized linear and logistic regression problems. Additionally, a linearized ADMM \citep{Li2020linearized} was proposed for matrix regression problems incorporating both fused lasso and nuclear norm penalties. \citet{Fan2019Generalized,10.1214/20-AOS1980} also implemented {a} contractive Peaceman-Rachford splitting method to solve their matrix optimization problem. Despite their effectiveness, these first-order methods suffer from
{slow convergence}, especially in the context of the large-scale data.
{Additionally, they often only leverage}
first-order information
{in the}
underlying non-smooth optimization model. In contrast, in recent study on vector-variate regression with lasso-type penalties \citep{li2017highly,li2018fused,Zhang_2018,luo2018solving,lin2019clustered,lin2023highly} and matrix-variate problem with nuclear or spectral norm  penalties \citep{Jiang2013Solving,jiang2014partial,Chen2014semismooth}, researchers have designed the semismooth Newton (SSN) based algorithms, which fully take advantage of {the} second-order {information} of the problem, to solve the optimization problem efficiently. These appealing evidences
{suggest the potential for designing a significantly}
more efficient algorithm {that} wisely exploits the inherent second-order {structured} sparsity of vector and special structure of matrix in the matrix regression problems.
{However,}
no research has yet leveraged this valuable insight to {address} problems involving both matrix and vector variables simultaneously. Thus, we aim to design a highly efficient second-order type, dual Newton method based proximal point algorithm (PPDNA) for solving the dual problem of the matrix regression problem \eqref{eq: p} with both matrix and vector variables.

In this paper, we propose a class of regularized regression models {in the form of} \eqref{eq: p} to estimate the matrix and vector predictors simultaneously. Our contributions are {threefold}. First, {our model accommodates a range of general}
convex penalty functions for $\phi$ and $\psi$, including {the nuclear norm for $\phi$, as well as the lasso, the fused lasso, and the group lasso for $\psi$.}
Second, we establish the $n$-consistency and $\sqrt{n}$-consistency of the estimator from \eqref{eq: p} with $\phi$ defined by the nuclear norm and $\psi$ defined by the $\ell_1$ norm. Third, we propose an efficient preconditioned proximal point algorithm for solving the general problem \eqref{eq: p}.

The rest of the paper is organized as follows. We establish the estimator consistency in Section~\ref{sec:consistency} and develop a preconditioned proximal point algorithm for solving the regularized regression problem \eqref{eq: p} in Section~\ref{sec:ppa}.
{Extensive numerical experiments are presented in Section~\ref{sec:numerical}. Finally, we} conclude the paper in Section~\ref{sec:conclusion}.

\vspace{2mm}
\noindent {\bf Notations and preliminaries.} The linear map ${\rm vec}: \mathbb{R}^{m\times q}\rightarrow \mathbb{R}^{mq}$ is defined as follows: for a matrix $Y \in \mathbb{R}^{m\times q}$, ${\rm vec}(Y) \in \mathbb{R}^{mq}$ is the column vector obtained by stacking the columns of $Y$ on top of one another. The linear map ${\rm mat}: \mathbb{R}^{mq}\rightarrow \mathbb{R}^{m\times q}$ is the {adjoint operator} of ${\rm vec}$, which is defined as follows: for a vector $y \in \mathbb{R}^{mq}$, ${\rm mat}(y)\in \mathbb{R}^{m\times q}$ is the  matrix obtained by reshaping the elements of $y$ back into a matrix with $m$ rows and $q$ columns, preserving the order of the elements. For $x \in \mathbb{R}$, {$x_{+}:=\max\{x,0\}$, and} ${\rm sgn}(x)$ denotes the sign function of $x$, that is ${\rm sgn}(x) = 1$ if $x>0$; ${\rm sgn}(x) = 0$ if $x=0$; ${\rm sgn}(x) = -1$ if $x<0$.
Let $h:\mathbb{R}^n \to (-\infty,\infty]$ be a closed proper convex function. The Moreau envelope of $h$ at $x$ is defined by
\begin{equation}
\begin{array}{l}
    {\rm E}_h(x) := \min\limits_{y\in\mathbb{R}^n}  h(y) + \frac{1}{2} \|y - x\|^2, \label{def: Moreau}
\end{array}
\end{equation}
and the corresponding proximal mapping ${\rm Prox}_h(x)$ is defined as the unique optimal solution of \eqref{def: Moreau}. It follows from  \citet{moreau1965proximite} and \citet{yosida1980functional} that for any $x\in\mathbb{R}^n$,
$
\nabla {\rm E}_h(x) = x - {\rm Prox}_h(x),
$
and ${\rm Prox}_h(\cdot)$ is Lipschitz continuous with modulus $1$. For $x\in \mathbb{R}^p$, let $\|x\|_1$ and $\|x\|$ be its $\ell_1$ norm and $\ell_2$ norm, respectively. For a matrix $X\in \mathbb{R}^{m\times q}$, we use $\|X\|_2,\|X\|_{F}$, $\|X\|_*$, and $\|X\|_{\infty}$ to denote the spectral norm, Frobenius norm, nuclear norm, and $\ell_{\infty}$ norm, respectively. For $f:\mathbb{R}^m\rightarrow \mathbb{R}^n$, $g:\mathbb{R}^p\rightarrow \mathbb{R}^m$, we denote their composition as $(f\circ g)(x) = f(g(x))$, $x\in \mathbb{R}^p$. Define $[n]:= \{1,\dots,n\}$. We denote $\odot$ as the Hadamard product, which is an element-wise multiplication of two matrices of the same size. Let $\mathbb{S}^n$ be the space of $n \times n$ symmetric matrices and $\mathbb{S}_{+}^n$ be the cone of symmetric positive semidefinite matrices. We use
$X\succeq Y$ to denote $X-Y\in \mathbb{S}_{+}^n$.
{For a locally Lipschitz continuous function} $\mathcal{F}:\,\mathbb{R}^n\to \mathbb{R}^m$, the Bouligand subdifferential (B-subdifferential) of $\mathcal{F}$  is denoted as
$\partial_B \mathcal{F}$, and the Clarke generalized Jacobian of $\mathcal{F}$  is denoted as
$\partial \mathcal{F}$.

\section{Consistency {Analysis}}\label{sec:consistency}
In this section, {we analyze the consistency and limiting distributions of our proposed estimator, through studying the asymptotic behavior of the objective function. As a representative example,} we consider the following matrix regression problem:
\begin{align}
\min_{U\in \mathbb{R}^{m\times q}, \beta \in \mathbb{R}^p}  \Big\{ Z_n(U,\beta):=\frac{1}{n}\sum_{i=1}^n \left(y_i \!-\!\langle X_i,U\rangle \!- \!\langle z_i,\beta\rangle\right)^2 +  \frac{\rho_n}{n}\|U\|_* + \frac{\lambda_n}{n}\|\beta\|_1
\Big\}.\label{eq: zn}
\end{align}
It is a special case of \eqref{eq: p} where the loss function $\hat{\ell}$ takes {the mean squared error}, the penalty {function} $\phi$ is proportional to the nuclear norm of the matrix variate, and the penalty {function} $\psi$ is proportional to the $\ell_1$ norm of the vector variate. For given $n$, $\rho_n$, and $\lambda_n$, we denote $(\widehat{B}_n,\hat{\gamma}_n)$ as a solution of \eqref{eq: zn}. We aim to analyze the consistency and limiting distribution of $(\widehat{B}_n,\hat{\gamma}_n)$.
We first give the following regularity assumptions. {It is worth noting that this assumption is similar to those commonly used in the literature to establish consistency  for various estimators, including the least absolute deviation estimator \cite{bassett1978asymptotic,pollard1991asymptotics}, Lasso-type penalized regression models \cite{knight2000asymptotics}, and regularized matrix regression estimators \cite{li2021double,wei2022asymptotics}.}

\begin{assumption}\label{assu: regularity}
We assume that the following regularity conditions hold:
\begin{align*}
&C_n:=\frac{1}{n}\sum_{i=1}^n {\rm vec}(X_i) {\rm vec}(X_i)^{\intercal}\rightarrow C,\quad
D_n:=\frac{1}{n}\sum_{i=1}^n z_i z_i^{\intercal}\rightarrow D, \quad
H_n:=\frac{1}{n}\sum_{i=1}^n {\rm vec}(X_i)z_i^{\intercal}\rightarrow H,
\end{align*}
where $C\in \mathbb{S}^{mq}_+$, $D\in \mathbb{S}^{p}_+$,
and $H\in \mathbb{R}^{mq\times p}$.
\end{assumption}

{Under the above assumption}, we further denote
\begin{align}\label{eq: S}
S_n:=\begin{pmatrix}
C_n & H_n\\
H_n^{\intercal} & D_n
\end{pmatrix}
\rightarrow
S := \begin{pmatrix}
C & H\\
H^{\intercal} & D
\end{pmatrix}.
\end{align}
{It can be seen} that $S_n $ and $S $ are positive semidefinite, since for any $u\in \mathbb{R}^{mq}$ and $ v\in \mathbb{R}^p$, it holds that
\begin{align*}
\begin{pmatrix}
u^{\intercal} & v^{\intercal}
\end{pmatrix}
S_n
\begin{pmatrix}
u\\
v
\end{pmatrix}
&=\frac{1}{n}\sum_{i=1}^n \left( u^{\intercal} {\rm vec}(X_i) {\rm vec}(X_i)^{\intercal} u + 2u^{\intercal} {\rm vec}(X_i) z_i^{\intercal} v + v^{\intercal} z_i z_i^{\intercal} v \right) \\
&=\frac{1}{n}\sum_{i=1}^n \left(u^{\intercal} {\rm vec}(X_i) + v^{\intercal} z_i  \right)^2\geq 0.
\end{align*}
{Note that when} the matrix $S_n$ is positive definite,  the estimator $(\widehat{B}_n,\hat{\gamma}_n)$ is unique.

The following theorem  shows the $n$-consistency of the estimator $(\widehat{B}_n,\hat{\gamma}_n)$. For penalty parameters $\rho_n = o(n)$ and $\lambda_n = o(n)$, the estimator $(\widehat{B}_n,\hat{\gamma}_n)$ converges to the true regression coefficients $(B,\gamma)$ in probability as $n\rightarrow \infty$. We outline the proof of Theorem~\ref{thm: n-consistency} here, with a detailed proof in Appendix~\ref{appendix: sec1}.

\begin{theorem}\label{thm: n-consistency}
Suppose Assumption \ref{assu: regularity} holds and  $S$ in \eqref{eq: S} is positive definite. If $\rho_n/n\rightarrow \rho_0\geq 0$ and $\lambda_n/n\rightarrow \lambda_0\geq 0$, then
\begin{align*}
(\widehat{B}_n,\hat{\gamma}_n) \rightarrow (\widehat{B},\hat{\gamma}):=\underset{U\in \mathbb{R}^{m\times q}, \beta \in \mathbb{R}^p}{\arg\min}\ Z(U,\beta)
\end{align*}
in probability as $n\rightarrow \infty$, where
\begin{align*}
Z(U,\beta) {:=} & \begin{pmatrix}
{\rm vec}(U- B)\\
\beta-\gamma
\end{pmatrix}^{\intercal}
S
\begin{pmatrix}
{\rm vec}(U- B)\\
\beta-\gamma
\end{pmatrix}
+  \rho_0\|U\|_* + \lambda_0\|\beta\|_1 .
\end{align*}
In particular, if $\rho_n = o(n)$ and $\lambda_n = o(n)$, then $\underset{U\in \mathbb{R}^{m\times q}, \beta \in \mathbb{R}^p}{\arg\min}\ Z(U,\beta) = (B,\gamma)$
and therefore $(\widehat{B}_n,\hat{\gamma}_n)$
is consistent.
\end{theorem}
\noindent \textbf{Sketch of proof.}
First, we show that $
Z_n(U,\beta) - Z(U,\beta)\rightarrow \sigma^2
$ in probability as $n\rightarrow \infty$. By the definitions of $Z_n$ and $Z$, we have
\begin{align*}
Z_n(U,\beta) \!- \!Z(U,\beta)
& =  {\rm vec}(U\!-\! B)^{\intercal} (C_n \!-\! C) {\rm vec}(U\!-\! B) + (\beta\!-\!\gamma)^{\intercal} (D_n\!-\!D) (\beta\!-\!\gamma) \\
&\qquad + 2  {\rm vec}(U- B)^{\intercal}  (H_n-H) (\beta-\gamma)   + \left(\frac{\rho_n}{n} -\rho_0\right)\|U\|_*  \\
&\qquad  + \left(\frac{\lambda_n}{n}- \lambda_0 \right)\|\beta\|_1 +\frac{1}{n}\sum_{i=1}^n \varepsilon_i^2 - \frac{2}{n}\sum_{i=1}^n  \varepsilon_i \left( \langle X_i,U-B\rangle +  \langle z_i,\beta-\gamma\rangle \right).
\end{align*}
{Then we} show that $\frac{1}{n}\sum_{i=1}^n  \varepsilon_i \left( \langle X_i,U-B\rangle +  \langle z_i,\beta-\gamma\rangle \right) \to 0$ in probability and therefore the first step completes.
By \citet{pollard1991asymptotics}, it further holds that, for any compact set $K\subset \mathbb{R}^{m\times q}\times \mathbb{R}^p$,
\begin{align}\label{eq: Z_converge_on_k_v1}
\sup_{(U,\beta)\in K} \ \Big| Z_n(U,\beta) - Z(U,\beta) - \sigma^2 \Big| \rightarrow 0
\mbox{ in probability as } n\rightarrow \infty.
\end{align}

Second, we show the sequence $\{(\widehat{B}_n,\hat{\gamma}_n)\}$ is  bounded in probability via showing the boundedness of $\{(\widehat{B}_n^{(0)},\hat{\gamma}_n^{(0)})\}$, defined by
\begin{align*}
(\widehat{B}_n^{(0)},\hat{\gamma}_n^{(0)}) :=\underset{U\in \mathbb{R}^{m\times q}, \beta \in \mathbb{R}^p}{\arg\min} \left\{
Z_n^{(0)}(U,\beta):=\frac{1}{n}\sum_{i=1}^n (y_i-\langle X_i,U\rangle - \langle z_i,\beta\rangle)^2 \right\}.
\end{align*}
The boundedness {of $\{(\widehat{B}_n,\hat{\gamma}_n)\}$} then ensures the applicability of \eqref{eq: Z_converge_on_k_v1}.

Finally, we establish the convergence $(\widehat{B}_n,\hat{\gamma}_n)\rightarrow (\widehat{B},\hat{\gamma})$ in probability by showing the convergence $Z(\widehat{B}_n,\hat{\gamma}_n) \rightarrow Z(\widehat{B},\hat{\gamma})$, leveraging the fact that $Z(\cdot,\cdot)$ is strongly convex.
\hfill $\blacksquare$

As {stated} in the above theorem, {$n$-consistency of the estimator $(\widehat{B}_n,\hat{\gamma}_n)$ requires $\rho_n = o(n)$ and $\lambda_n = o(n)$.} However, for $\sqrt{n}$-consistency {of $(\widehat{B}_n,\hat{\gamma}_n)$}, we require penalty parameters grow more slowly: $\rho_n = o(\sqrt{n})$, $\lambda_n = o(\sqrt{n})$, as shown in the following theorem.

\begin{theorem}
Suppose Assumption \ref{assu: regularity} holds and  $S$ in \eqref{eq: S} is positive definite. If $\rho_n/\sqrt{n}\rightarrow \rho_0\geq 0$ and $\lambda_n/\sqrt{n}\rightarrow \lambda_0\geq 0$, then
\begin{align*}
\sqrt{n} (\widehat{B}_n-B,\hat{\gamma}_n-\gamma) \rightarrow \underset{U\in \mathbb{R}^{m\times q}, \beta \in \mathbb{R}^p}{\arg\min}\ F(U,\beta)
\end{align*}
in distribution as $n\rightarrow \infty$, where
\begin{align*}
&F(U,\beta)  := {\rm vec}( U)^{\intercal} C {\rm vec}( U) + \beta^{\intercal} D \beta + 2 {\rm vec}( U)^{\intercal}  H \beta - 2 {\rm vec}( U)^{\intercal} w_1 - 2\beta^{\intercal} w_2   \\
& \quad + \rho_0 \Big(\sum_{i=1}^r   P_{\cdot i}^{\intercal} U Q_{\cdot i}  + \sum_{i = r + 1}^{\min\{m,q\}} |P_{\cdot i}^{\intercal} U Q_{\cdot i}| \Big) + \lambda_0 \sum_{i=1}^p \Big( \beta_i \, {\rm sgn}(\gamma_i)   I(\gamma_i\neq 0) + |\beta_i| I(\gamma_i=0) \Big).
\end{align*}
Here, $I({\rm Event}) = 1$ if Event happens, and $I({\rm Event}) =0$ otherwise. $\begin{pmatrix}
w_1\\
w_2
\end{pmatrix}\sim\mathcal{N}(0,\sigma^2 S)$ is a random variable in $\mathbb{R}^{mq+p}$.
$B = P\Sigma Q^{\intercal}$ is the singular value decomposition of $B$, where $\Sigma\in \mathbb{R}^{m\times q}$ is a rectangular diagonal matrix with diagonal elements $\alpha_1 \geq \cdots \geq \alpha_r > \alpha_{r+1} = \cdots = \alpha_{\min\{m,q\}} = 0$, $r = {\rm rank}(B)$,
$P = (P_{\cdot 1},\cdots,P_{\cdot m})\in \mathbb{R}^{m\times m}$, and $Q = (Q_{\cdot 1},\cdots,Q_{\cdot q})\in \mathbb{R}^{q\times q}$ are orthogonal matrices.
\\[2mm]
In particular, if $\rho_n = o(\sqrt{n})$ and $\lambda_n = o(\sqrt{n})$, then $\underset{U\in \mathbb{R}^{m\times q}, \beta \in \mathbb{R}^p}{\arg\min}\ F(U,\beta) = (U^*,\beta^*)$
with {$[
{\rm vec}(U^*);
\beta^*
] = S^{-1}[
w_1;
w_2
]\sim\mathcal{N}(0,\sigma^2 S^{-1}).
$}
\end{theorem}
\begin{proof}
We define
\begin{align*}
F_n(U,\beta)&:=\sum_{i=1}^n \left[ \Big(\varepsilon_i- \frac{1}{\sqrt{n}}\langle X_i,U\rangle -\frac{1}{\sqrt{n}}\langle z_i,\beta \rangle \Big)^2 -\varepsilon_i^2 \right]\\
&\quad + \rho_n\left( \|B+\frac{1}{\sqrt{n}}U\|_* - \|B\|_*\right) + \lambda_n\left( \|\gamma+\frac{1}{\sqrt{n}}\beta\|_1 - \|\gamma\|_1\right),
\end{align*}
for any $(U,\beta)\in \mathbb{R}^{m\times q}\times \mathbb{R}^p$, which is minimized at $\sqrt{n} (\widehat{B}_n-B,\hat{\gamma}_n-\gamma)$ by the definition of $(\widehat{B}_n,\hat{\gamma}_n)$  and  the fact that
\begin{align*}
\Big(\varepsilon_i- \frac{1}{\sqrt{n}}\langle X_i,U\rangle -\frac{1}{\sqrt{n}}\langle z_i,\beta \rangle \Big)^2
&= \Big(y_i - \langle X_i,B + \frac{1}{\sqrt{n}}U \rangle - \langle z_i,\gamma+\frac{1}{\sqrt{n}}\beta \rangle \Big)^2.
\end{align*}

By direct computations, we have that
\begin{align*}
& \frac{1}{\sqrt{n}}\sum_{i=1}^n \varepsilon_i \left( \langle X_i,U\rangle +  \langle z_i,\beta \rangle \right) = \begin{pmatrix}
{\rm vec}(U)^{\intercal}  &
 \beta^{\intercal}
\end{pmatrix}
\begin{pmatrix}
 \frac{1}{\sqrt{n}}\sum_{i=1}^n  \varepsilon_i {\rm vec}(X_i)\\
 \frac{1}{\sqrt{n}}\sum_{i=1}^n  \varepsilon_i z_i
\end{pmatrix}, \mbox{ where} \\
& \mathbb{E}\left[
\begin{pmatrix}
\frac{1}{\sqrt{n}}\sum_{i=1}^n  \varepsilon_i {\rm vec}(X_i)\\
\frac{1}{\sqrt{n}}\sum_{i=1}^n  \varepsilon_i z_i
\end{pmatrix}
\right] = 0, \ \mbox{ and }\
{\rm Var}\left[
\begin{pmatrix}
 \frac{1}{\sqrt{n}}\sum_{i=1}^n  \varepsilon_i {\rm vec}(X_i)\\
\frac{1}{\sqrt{n}}\sum_{i=1}^n  \varepsilon_i z_i
\end{pmatrix}
\right] = \sigma^2 S_n.
\end{align*}

With the above results, the first term of $F_n$
\begin{align*}
 & \sum_{i=1}^n \left[ \Big(\varepsilon_i- \frac{1}{\sqrt{n}}\langle X_i,U\rangle -\frac{1}{\sqrt{n}}\langle z_i,\beta \rangle \Big)^2 -\varepsilon_i^2 \right] \\
= & \frac{1}{n} \!\sum_{i=1}^n \!(\langle X_i,U \rangle)^2 \!\!+\! \frac{1}{n} \!\sum_{i=1}^n \! (\langle z_i,\beta \rangle)^2 \!\! +\! \frac{2}{n} \!\sum_{i=1}^n\! \langle X_i,U\rangle\langle z_i,\beta \rangle
\!-\! \frac{2}{\sqrt{n}}\sum_{i=1}^n \!\varepsilon_i\! \left( \langle X_i,U\rangle \!\!+\!\! \langle z_i,\beta \rangle \right) \\
=  & {\rm vec}(U)^{\intercal} C_n {\rm vec}(U) + \beta^{\intercal} D_n \beta + 2 {\rm vec}(U)^{\intercal} H_n \beta - \frac{2}{\sqrt{n}}\sum_{i=1}^n \varepsilon_i \left( \langle X_i,U\rangle + \langle z_i,\beta \rangle \right) \\
\rightarrow  & {\rm vec}( U)^{\intercal} C {\rm vec}( U) + \beta^{\intercal} D \beta + 2 {\rm vec}( U)^{\intercal}  H \beta - 2 {\rm vec}( U)^{\intercal} w_1 - 2\beta^{\intercal} w_2
\end{align*}
in distribution. Here $\begin{pmatrix}
    w_1\\w_2
\end{pmatrix} \!\!\sim\!\mathcal{N}(0,\!\sigma^2 S)$ is the limiting distribution of the random vector
$\begin{pmatrix}
 \!\frac{1}{\sqrt{n}}\!\!\sum_{i=1}^n \! \varepsilon_i {\rm vec}(X_i)\\
 \frac{1}{\sqrt{n}}\sum_{i=1}^n  \varepsilon_i z_i
\end{pmatrix}$
as $n\rightarrow \infty$. For the remaining terms of $F_n$, we have
\begin{align*}
&\lim_{n \rightarrow \infty} \rho_n\left[ \|B+\frac{1}{\sqrt{n}}U\|_* - \|B\|_*\right]
= \rho_0 \Big( \sum_{i=1}^r   P_{\cdot i}^{\intercal} U Q_{\cdot i}  + \sum_{i = r + 1}^{\min\{m,q\}} |P_{\cdot i}^{\intercal} U Q_{\cdot i}| \Big),\\
&\lim_{n \rightarrow \infty} \lambda_n\left[ \|\gamma+\frac{1}{\sqrt{n}}\beta\|_1 - \|\gamma\|_1\right] = \lambda_0 \sum_{i=1}^p \Big( \beta_i\,  {\rm sgn}(\gamma_i)  I(\gamma_i\neq 0) + |\beta_i| I(\gamma_i=0) \Big),
\end{align*}
where the first equation follows from \citet[Theorem 1]{watson1992characterization}:
\begin{align*}
&\lim_{n \rightarrow \infty} \frac{\rho_n}{\sqrt{n}} \frac{\|B + \frac{1}{\sqrt{n}}U\|_* - \|B\|_*}{1/\sqrt{n}}
=\rho_0  \max_{d\in \partial \|\alpha\|_1 } \sum_{i=1}^{\min\{m,q\}} d_iP_{\cdot i}^{\intercal} U Q_{\cdot i}\\
= \,\, & \rho_0 \sum_{i=1}^{\min\{m,q\}} \max_{d_i\in \partial|\alpha_i| } d_iP_{\cdot i}^{\intercal} U Q_{\cdot i}
= \rho_0 \Big(\sum_{i=1}^r   P_{\cdot i}^{\intercal} U Q_{\cdot i}  + \sum_{i = r + 1}^{\min\{m,q\}} |P_{\cdot i}^{\intercal} U Q_{\cdot i}| \Big),
\end{align*}
and the second equation follows from direct computations.
Therefore, we have proved that
$ F_n (U,\beta) \rightarrow F(U,\beta) $
in distribution as $n\rightarrow \infty$.
Lastly, given that $F(\cdot,\cdot)$ is  strongly convex with $S\succ 0$, guaranteeing a unique minimizer, and $F_n(\cdot,\cdot)$ is convex, we deduce from the work of \citet{geyer1994asymptotics} that
$\sqrt{n} (\widehat{B}_n-B,\hat{\gamma}_n-\gamma) \rightarrow \underset{U\in \mathbb{R}^{m\times q}, \beta \in \mathbb{R}^p}{\arg\min} \ F(U,\beta)$
in distribution as $n\rightarrow \infty$.
In particular, when $\rho_0=\lambda_0=0$, the minimizer of $F$ is given by
\begin{align*}
  (U^*,\beta^*) := \underset{U\in \mathbb{R}^{m\times q}, \beta \in \mathbb{R}^p}{\arg\min}  F(U,\beta)  \ \mbox{with} \
  \begin{pmatrix}
{\rm vec}(U^*)\\
\beta^*
\end{pmatrix} = S^{-1}\begin{pmatrix}
w_1\\
w_2
\end{pmatrix}\sim \mathcal{N}(0,\sigma^2 S^{-1}).
\end{align*}
The proof is completed.
\end{proof}

\section{A Preconditioned Proximal Point Algorithm}\label{sec:ppa}
In this section, we design a preconditioned proximal point algorithm (PPA) for solving problem \eqref{eq: p}. Starting from an initial point $(B^0,\gamma^0)\in \mathbb{R}^{m\times q}\times \mathbb{R}^p$, the preconditioned PPA iteratively computes a sequence $\{(B^k,\gamma^k)\}\subseteq \mathbb{R}^{m\times q}\times \mathbb{R}^p$ as follows:
\begin{align}
(B^{k+1},\gamma^{k+1})\approx \mathbb{J}_k(B^k,\gamma^k):=\underset{B\in \mathbb{R}^{m\times q} ,{\gamma\in\mathbb{R}^p}}{\arg\min} G_k(B,\gamma),\label{eq: ppa_k}
\end{align}
where
\begin{align*}
G_k(B,\gamma) :=G(B,\gamma) + \frac{1}{2\sigma_k}(\|B-B^k\|^2 + \|\gamma-\gamma^k\|^2 + \nu \|X {\rm vec}(B) + Z \gamma - X {\rm vec} (B^k) - Z \gamma^k\|^2),
\end{align*}
$\{\sigma_k\}$ is a sequence of nondecreasing positive real numbers, and $\nu > 0$ is a scaling parameter. When $\nu = 0$, it {reduces} to the classic PPA \citep{rockafellar1976monotone}. We say \eqref{eq: ppa_k} is a preconditioned PPA iteration due to the last term of $G_k$. This term is inevitable to derive a smooth unconstrained dual problem of  \eqref{eq: ppa_k}, as we will demonstrate shortly. Clearly, an efficient solver for \eqref{eq: ppa_k} is crucial for the effective implementation of the preconditioned PPA.

We let {the convex function $h:\mathbb{R}^n\rightarrow \mathbb{R}$ be defined as}
\begin{align}\label{eq: fct_h}
    h(s):=\hat{\ell}(s,y)
\end{align}
and $s^k:=X {\rm vec}(B^k) + Z \gamma^k$, the minimization problem in \eqref{eq: ppa_k} can be written equivalently as
\begin{equation}\label{eq: c_ppa_k}
\begin{aligned}
\min_{B\in \mathbb{R}^{m\times q},\gamma\in \mathbb{R}^p,s\in \mathbb{R}^n} &
h(s) \!+\! \phi(B) \!+\! \psi(\gamma) \!+ \!\frac{1}{2\sigma_k}\|B\!\!-\!\!B^k\|^2 \!+ \!\frac{1}{2\sigma_k}\|\gamma\!\!-\!\!\gamma^k\|^2 \!+\! \frac{\nu}{2\sigma_k}\|s \!\!-
\!\! s^k\|^2
\\
{\rm s.t.} \quad\quad\quad & \ X {\rm vec}(B) + Z \gamma -s = 0.
\end{aligned}
\end{equation}
The Lagrangian function associated with \eqref{eq: c_ppa_k} is
\begin{align*}
&\mathcal{L}(B,\gamma,s;\xi)
=\  h(s) + \phi(B) + \psi(\gamma) + \langle \xi,X {\rm vec}(B) + Z \gamma -s\rangle  \\
&\qquad \qquad \qquad \quad + \frac{1}{2\sigma_k}\|B-B^k\|^2 + \frac{1}{2\sigma_k}\|\gamma-\gamma^k\|^2 + \frac{\nu}{2\sigma_k}\|s - s^k\|^2  \\
= & h(s) + \phi(B) + \psi(\gamma) + \frac{1}{2\sigma_k} \left( \|B^k\|^2  + \| \gamma^k\|^2 + \nu \|s^k\|^2 \right)\\
&  + \frac{1}{2\sigma_k}\|B\!-\!B^k\!+\! \sigma_k {\rm mat} (X^{\intercal}\xi) \|^2 + \frac{1}{2\sigma_k}\|\gamma \!-\!\gamma^k \!+\!\sigma_k Z^{\intercal}\xi\|^2 + \frac{\nu}{2\sigma_k}\|s \!-\! s^k \!-\!\frac{\sigma_k}{\nu} \xi\|^2 \\
&   -\frac{1}{2\sigma_k} \left( \|B^k - \sigma_k {\rm mat} (X^{\intercal}\xi)\|^2  + \|\gamma^k -\sigma_k Z^{\intercal}\xi\|^2  + \nu \|s^k +\frac{\sigma_k}{\nu} \xi\|^2 \right) ,
\end{align*}
for $ (B,\gamma,s,\xi)\in \mathbb{R}^{m\times q}\times \mathbb{R}^p\times \mathbb{R}^n\times \mathbb{R}^n$.
Therefore, the Lagrangian dual problem of \eqref{eq: c_ppa_k} takes the form of
\begin{align}
\max_{\xi \in \mathbb{R}^n} \ \left\{  \Phi_k(\xi):= \min_{B\in \mathbb{R}^{m\times q},\gamma\in \mathbb{R}^p,s\in \mathbb{R}^n} \ \mathcal{L}(B,\gamma,s;\xi) \right\}.\label{eq: ppa_k_d}
\end{align}
{By} the definition of Moreau envelope \eqref{def: Moreau}, {it can be derived that}
\begin{align*}
 \Phi_k(\xi) = \,&
\frac{1}{\sigma_k} \bigg[ {\rm E}_{\sigma_k \phi}(B^k \!-\! \sigma_k  {\rm mat}(X^{\intercal}\xi)) + {\rm E}_{\sigma_k \psi}(\gamma^k \!-\!\sigma_k Z^{\intercal}\xi) +  \nu  {\rm E}_{\sigma_k h/\nu} (s^k \!+\!\frac{\sigma_k}{\nu} \xi)  \\
& -\frac{1}{2} \left( \|B^k - \sigma_k {\rm mat} (X^{\intercal}\xi)\|^2  + \|\gamma^k -\sigma_k Z^{\intercal}\xi\|^2  + \nu \|s^k +\frac{\sigma_k}{\nu} \xi\|^2 \right)\\
&+ \frac{1}{2} \left( \|B^k\|^2  + \| \gamma^k\|^2 + \nu \|s^k\|^2 \right) \bigg].
\end{align*}
{An exciting observation} is that $\Phi(\cdot)$ is continuously differentiable {with Lipschitz continuous gradient}, and therefore we can solve problem \eqref{eq: ppa_k_d} efficiently via a non-smooth Newton type method for solving the equation $\nabla \Phi_k(\xi)=0$; see Section~\ref{sec:ssn} for details.
In addition, the Karush-Kuhn-Tucker (KKT) optimality conditions associated with \eqref{eq: c_ppa_k} and \eqref{eq: ppa_k_d} are
\begin{align*}
&B = {\rm Prox}_{\sigma_k \phi}(B^k - \sigma_k {\rm mat}(X^{\intercal}\xi)), \quad \gamma = {\rm Prox}_{\sigma_k \psi}(\gamma^k -\sigma_k Z^{\intercal}\xi),\\
&s = {\rm Prox}_{\sigma_k h/\nu} (s^k +\frac{\sigma_k}{\nu} \xi), \quad X {\rm vec}(B) + Z \gamma -s = 0.
\end{align*}
Once we obtain an approximate dual solution $\xi^{k+1}$ to \eqref{eq: ppa_k_d}, we can construct an approximate primal solution $(B^{k+1},\gamma^{k+1},s^{k+1})$ to \eqref{eq: c_ppa_k} based on the above KKT conditions; see  \eqref{eq:ppa_update_primal}.

\begin{algorithm}
\caption{A preconditioned proximal point algorithm for solving \eqref{eq: p}}
\label{alg:algorithm1}
\begin{algorithmic}[1]
\REQUIRE $\nu > 0$, $0 < \sigma_0 \leq \sigma_{\infty} < \infty$, $\epsilon_k \geq 0$ with $\sum_{k=0}^{\infty} \epsilon_k < \infty$;

\ENSURE an approximate optimal solution $(B^{k+1},\gamma^{k+1})$ to \eqref{eq: p}.

\STATE \textbf{Initialization}: choose $B^0 \in \mathbb{R}^{m\times q}$, $\gamma^0 \in \mathbb{R}^{p}$, $k = 0$.
\REPEAT
\STATE \textbf{Step 1.} Solve via Algorithm~\ref{alg:semismoothNewton}
\begin{equation}
    \xi^{k+1}\approx \arg\max{\Phi_{k}(\xi)}.
    \label{eq: ppa_sub}
\end{equation}
\STATE \textbf{Step 2.} Compute
\begin{align}\label{eq:ppa_update_primal}
\left\{
\begin{aligned}
&B^{k+1} = {\rm Prox}_{\sigma_k \phi}(B^k - \sigma_k {\rm mat}(X^{\intercal}\xi^{k+1})),\\
&\gamma^{k+1} = {\rm Prox}_{\sigma_k \psi}(\gamma^k -\sigma_k Z^{\intercal}\xi^{k+1}),\\
&s^{k+1} = {\rm Prox}_{\sigma_k h/\nu} (s^k +\frac{\sigma_k}{\nu} {\xi^{k+1}}).
\end{aligned}
\right.
\end{align}
\STATE \textbf{Step 3.} Update $\sigma_{k+1} \in [\sigma_k, \sigma_{\infty}]$, $k \leftarrow k + 1$.
\UNTIL Stopping criterion is satisfied.
\end{algorithmic}
\end{algorithm}

The preconditioned PPA for solving \eqref{eq: p} is given in Algorithm \ref{alg:algorithm1}.  For controlling the inexactness in the subproblem \eqref{eq: ppa_sub}, we use the following implementable stopping condition \eqref{eq: stopA} based on the duality gap between the subproblem \eqref{eq: ppa_k} and its dual problem \eqref{eq: ppa_k_d}:
\begin{align}\label{eq: stopA}
G_k(B^{k+1},\gamma^{k+1}) - \Phi_k(\xi^{k+1}) \leq\frac{\epsilon_k^2}{2\sigma_k},
\quad \epsilon_k \geq 0,
\quad \sum_{k=0}^{\infty} \epsilon_k < \infty.
\end{align}

For completeness, we present the global convergence of Algorithm~\ref{alg:algorithm1} in Theorem~\ref{thm:convergence}. Its proof follows a similar approach to that of \citet[Theorem 2.1]{lin2023highly}, using results of \citet[Exercise 8.8]{RockafellarWets1998} and \citet[Theorem 2.3]{li2020}.

\begin{theorem}\label{thm:convergence}
Assume the optimal solution set of
problem \eqref{eq: p} is nonempty, denoted as $\Omega$.
Let $\left\{\left(B^k,\gamma^k,\xi^k\right)\right\}$ be the sequence generated by Algorithm~\ref{alg:algorithm1}, where the stopping criterion \eqref{eq: stopA} is satisfied at {\bf Step 1} of each iteration when solving \eqref{eq: ppa_sub}.  Then the sequence $\left\{(B^k,\gamma^k)\right\}$ is bounded and converges to some $(B^*,\gamma^*) \in \Omega$.
\end{theorem}

\subsection{A Semismooth Newton Method for Solving the PPA Subproblem}\label{sec:ssn}

Efficiently solving the subproblem \eqref{eq: ppa_sub} poses a significant implementation challenge for Algorithm~\ref{alg:algorithm1}. Our goal is to devise an efficient semismooth Newton (SSN) method that exhibits at least superlinear convergence. For more details about semismoothness and SSN, see Appendix~\ref{appendix: sec2}. {Note that problem \eqref{eq: ppa_sub}} can be computed by solving the non-smooth equation $ \nabla \Phi_k(\xi)=0$, where

\[
\setlength{\abovedisplayskip}{1pt}  
  \setlength{\belowdisplayskip}{2pt}
\nabla \Phi_k(\xi) =  X {\rm vec}({\rm Prox}_{\sigma_k \phi}(B^k - \sigma_k {\rm mat}( X^{\intercal}\xi)))
+ Z {\rm Prox}_{\sigma_k \psi}(\gamma^k -\sigma_k Z^{\intercal}\xi) - {\rm Prox}_{\sigma_k h/\nu} \left(s^k +\frac{\sigma_k}{\nu} \xi\right).
\]

Define the multifunction $\hat{\partial }^2\Phi_k$ {as: for any $\xi \in \mathbb{R}^n$, $\hat{\partial }^2\Phi_k(\xi)$ is a multifunction from $\mathbb{R}^{n}$ to $\mathbb{R}^{n}$ such that for each $u\in \mathbb{R}^n$,
\begin{align}
\hat{\partial }^2\Phi_k(\xi)[u] \!\!:=\!\!
\left\{ \begin{aligned}
& \!\!-\!\!\sigma_k X {\rm vec}\left( \mathcal{W} \left( {\rm mat}(X^{\intercal}u)\right)\right) \\
&-\sigma_k Z \mathcal{Q} Z^{\intercal}u - \frac{\sigma_k}{\nu} \mathcal{M}u
\end{aligned} \ \middle\vert
\begin{aligned}
& \mathcal{M}\in \hat{\partial} {\rm Prox}_{\sigma_k h/\nu} (s^k \!+\!\frac{\sigma_k}{\nu} \xi)\\
&\mathcal{W}\in \hat{\partial} {\rm Prox}_{\sigma_k \phi}(B^k \!-\! \sigma_k {\rm mat}( X^{\intercal}\xi))\\
&\mathcal{Q}\in \hat{\partial} {\rm Prox}_{\sigma_k \psi}(\gamma^k \!-\!\sigma_k Z^{\intercal}\xi)
\end{aligned} \right\},\label{eq: jacobian}
\end{align}
where ${\rm vec}$ and ${\rm mat}$ are linear maps as defined at the end of the Section \ref{sec:intro}.} Here, we introduce $\hat{\partial} {\rm Prox}_f$ as a more computationally efficient alternative to the Clarke generalized Jacobian ${\partial} {\rm Prox}_f$, or even itself, associated with a given closed proper convex function $f:\mathbb{R}^n\rightarrow (-\infty,\infty]$, while ensuring Property A.
A multifunction $\hat{\partial} {\rm Prox}_{f}$ is said to satisfy Property A with respect to $ {\rm Prox}_{f}$, if (1) it is nonempty, compact valued, and upper-semicontinuous; (2) for any $x\in \mathbb{R}^n$, all the elements in $\hat{\partial} {\rm Prox}_{f}(x)$ are symmetric and positive semidefinite; (3) $ {\rm Prox}_{f}$ is strongly semismooth with respect to $\hat{\partial} {\rm Prox}_{f}$.

We provide the framework of the SSN for solving \eqref{eq: ppa_sub} in Algorithm~\ref{alg:semismoothNewton}.
In practice, one can choose the parameters as $\mu=10^{-4}, \tau=0.5, \eta=0.005$ and $\delta=0.5$.

\begin{algorithm}
\caption{A semismooth Newton method for solving \eqref{eq: ppa_sub}}
\label{alg:semismoothNewton}
\begin{algorithmic}[1]
\REQUIRE $\mu \in (0,0.5)$, $\tau \in (0,1]$,  $\eta\in (0,1)$, $\delta \in (0,1)$.
\ENSURE an approximate optimal solution $\xi^{k+1}$ to \eqref{eq: ppa_sub}.
\STATE \textbf{Initialization}: choose $\xi^{k,0} \in \mathbb{R}^n$, $j = 0$.
\REPEAT
\STATE \textbf{Step 1.} Select an element $\mathcal{H}_j \in \hat{\partial}^2 \Phi_k (\xi^{k,j})$. Apply the direct method or the conjugate gradient (CG) method to find an approximate solution $d^j \in \mathbb{R}^n$ to
\begin{equation}
    \mathcal{H}_jd^j \approx -\nabla \Phi_k(\xi^{k,j}), \label{eq: sub_ssn}
\end{equation}
such that
$\displaystyle
\|\mathcal{H}_jd^j + \nabla \Phi_k(\xi^{k,j})\| \leq \min(\eta, \|\nabla \Phi_k(\xi^{k,j})\|^{1+\tau}).
$
\STATE \textbf{Step 2.} Set $\alpha_j = \delta^{m_j}$, where $m_j$ is the smallest nonnegative integer $m$ for which
\[
\Phi_k(\xi^{k,j} + \delta^m d^j) \geq \Phi_k(\xi^{k,j}) + \mu \delta^m \langle \nabla \Phi_k(\xi^{k,j}), d^j \rangle.
  \setlength{\belowdisplayskip}{1pt}
\]
\STATE \textbf{Step 3.} Set $\xi^{k,j+1} = \xi^{k,j} + \alpha_j d^j$, $\xi^{k+1} = \xi^{k,j+1}$, $j \leftarrow j + 1$.
\UNTIL Stopping criterion
is satisfied.
\end{algorithmic}
\end{algorithm}

Theorem~\ref{thm:ssn_convergence} gives the superlinear convergence of the SSN method, which can be proved using the results of \citet[Proposition 3.3, Theorems 3.4 and 3.5]{zhao2010newton} and \citet[Theorem 4]{li2017highly}. For simplicity, we omit the proof here.

\begin{theorem}\label{thm:ssn_convergence}
Assume that the function $h$ in \eqref{eq: fct_h} is twice continuously differentiable, and
$\hat{\partial} {\rm Prox}_{\sigma_k \phi}$ and $\hat{\partial} {\rm Prox}_{\sigma_k \psi}$ satisfy Property A with respect to $ {\rm Prox}_{\sigma_k \phi}$ and $ {\rm Prox}_{\sigma_k \psi}$, respectively.
Let $\left\{\xi^{k, j}\right\}$ be the sequence generated by Algorithm~\ref{alg:semismoothNewton}. Then $\left\{\xi^{k, j}\right\}$ converges globally to the unique optimal solution
$\bar{\xi}^{k+1}$ of problem \eqref{eq: ppa_sub}.
Furthermore, the convergence rate is at least superlinear, i.e., for $j$ sufficiently large,
$
\left\|\xi^{k, j+1}-\bar{\xi}^{k+1}\right\|=O\left(\left\|\xi^{k, j}-\bar{\xi}^{k+1}\right\|^{1+{\tau}}\right),
$
where ${\tau} \in(0,1]$ is a given parameter in Algorithm~\ref{alg:semismoothNewton}.
\end{theorem}

\begin{remark}
Here are some comments on
Theorem~\ref{thm:ssn_convergence}. First, when the B-subdifferential or Clarke generalized Jacobian of ${\rm Prox}_{\sigma_k \phi}$ (${\rm Prox}_{\sigma_k h/\nu}$, ${\rm Prox}_{\sigma_k \psi}$) is easy to derive, we can simply take $\hat{\partial} {\rm Prox}_{\sigma_k \phi} = \partial_B {\rm Prox}_{\sigma_k \phi}$ or $\hat{\partial} {\rm Prox}_{\sigma_k \phi} = \partial {\rm Prox}_{\sigma_k \phi}$.
Second, under the assumption that  $h$ is twice continuously differentiable, by \citet[Proposition 4.1]{lin2023highly}, we will have the strict concavity of the dual objective function $\Phi_k(\cdot)$, leading to the existence of a unique maximizer for problem \eqref{eq: ppa_sub} and the negative definiteness of all elements in $\hat{\partial }^2\Phi_k(\bar{\xi}^{k+1})$. In fact, we can also obtain similar convergence results when $h$ is a square root loss  under mild assumptions, see \cite{tang2020sparse}. Moreover, for twice continuously differentiable $h$, $\hat{\partial}{\rm Prox}_{\sigma_k h/\nu}$ reduces to one element $ \nabla {\rm Prox}_{\sigma_k h/\nu}$, and it naturally satisfies
Property A.
\end{remark}

For many commonly used penalty functions $\phi$ and $\psi$, their proximal mappings are strongly semismooth  with respect to their (surrogate) generalized Jacobian. We illustrate the application of our algorithm to a specific problem in the following subsection.

\subsection{Application on Low Rank and Sparse Regularized Matrix Regression Problems}
In this subsection, we consider the application of our algorithmic framework on  specific low-rank and sparse regularized matrix regression problems, where
\begin{align*}
h(u) \!=\! \frac{1}{2}\|u-y\|^2,\ u\in \mathbb{R}^n,\ \ \phi(B)\!=\! \rho \|B\|_*, \ B\in \mathbb{R}^{m\times q}, \ \  \psi(\gamma) \!=\! \lambda\|\gamma\|_1,\ \gamma \in \mathbb{R}^p.
\end{align*}
We emphasize that our algorithmic framework is quite general as long as
{the explicit formulae of $ {\rm Prox}_{\phi}$ and $ {\rm Prox}_{\psi}$
are available and their corresponding (surrogate) generalized Jacobians can be well constructed.}
The following list provides some examples of $\psi(\cdot)$, whose proximal mapping ${\rm Prox}_{\psi}(\cdot)$ along with the (surrogate) generalized Jacobian $\hat{\partial} {\rm Prox}_{\psi}(\cdot)$ are available in the literature.
Moreover, the proximal mapping ${\rm Prox}_{\psi}(\cdot)$ is shown to be strongly semismooth, with respect to {its} (surrogate) generalized Jacobian $\hat{\partial} {\rm Prox}_{\psi}(\cdot)$, enabling their integration into our algorithmic framework.

\begin{itemize}[itemsep=1pt, topsep=2pt, partopsep=0pt, parsep=0pt]
    \item Lasso regularizer \citep{tibshirani1996regression}:\\ $\psi(\gamma)=\lambda\|\gamma\|_1,\ \lambda > 0$.
The proximal mapping and its generalized Jacobian has been well-studied by  \citet[Section 3.3]{li2017highly}.

\item Fused lasso regularizer \citep{fusedlasso}:\\
$ \psi(\gamma)=\lambda \|\gamma\|_1 + \lambda' \sum_{i=1}^{p-1}|\gamma_i -\gamma_{(i+1)}|, \ \lambda > 0,\ \lambda' > 0$.
The proximal mapping and its generalized Jacobian has been well-studied by \citet[Section 3]{li2018fused}.

\item Sparse group lasso regularizer \citep{sgl1,sgl2}:\\
$
\psi(\gamma)=\lambda\|\gamma\|_1+\lambda' \sum_{l=1}^g w_l\left\|\gamma_{G_l}\right\|, \
\lambda>0,\ \lambda'>0, \ w_1, \dots, w_g \geq 0$, and $\left\{G_1, \cdots, G_g\right\}$ is a disjoint partition of the set $[p]$. It reduces to the group lasso regularizer when $\lambda=0$.
The proximal mapping and its generalized Jacobian has been well-studied by \citet[Proposition 2.1 \& Section 3]{Zhang_2018}.

\item Clustered lasso regularizer \citep{clusteredlasso1, clusteredlasso2}:\\
$\psi(\gamma)=\lambda\|\gamma\|_1+\lambda' \sum_{1 \leq i<j \leq p}\left|\gamma_i-\gamma_j\right|,\lambda>0,\lambda'>0$.  The proximal mapping and its generalized Jacobian has been well-studied by \citet[Section 2]{lin2019clustered}.

\item Exclusive lasso regularizer \citep{exclusivelasso1,exclusivelasso2}:\\
$
\psi(\gamma)=\lambda \sum_{l=1}^g\left\|w_{G_l} \odot \gamma_{G_l}\right\|_1^2, \ \lambda>0$,
$w=(w_{G_1}, \cdots, w_{G_g}) \in {\mathbb{R}_{++}^p}$ is a weight vector
and $\left\{G_1, \cdots, G_g\right\}$ is a disjoint partition of the index set $[p]$. The proximal mapping and its generalized Jacobian has been well studied by \citet[Section 3]{lin2023highly}.

\item Decreasing
weighted sorted $\ell_1$-norm (DWSL1) regularizer, also known as sorted $\ell$-one penalized estimation (SLOPE) regularizer \citep{slope}:\\
$
\psi(\gamma)=\sum_{i=1}^p\mu_i|\gamma|_{(i)}, \mu_1 \geq \cdots \geq \mu_p \geq 0$ and $\mu_1>0$. Here $|\gamma|_{(i)}$ is the $i$-th largest component of $|\gamma|$ such that $|\gamma|_{(1)} \geq \cdots \geq|\gamma|_{(p)}$ and $|\gamma|\in \mathbb{R}^p$ is   obtained from $\gamma$ by taking the absolute value of its components. The proximal mapping and its generalized Jacobian has been well-studied by \citet[Section 2]{luo2018solving}. Note that Octagonal Shrinkage and Clustering Algorithm for
Regression (OSCAR) regularizer $\psi(\gamma)=\lambda\|\gamma\|_1+\lambda' \max _{i<j}\left\{\left|\gamma_i\right|,\left|\gamma_j\right|\right\}$ is a special case of the DWSL1 regularizer.
\end{itemize}

Without loss of generality, we assume that $m \leq q$. The next proposition summarizes the formulae for the proximal mapping ${\rm Prox}_{\rho \|\cdot\|_*}$ and its Clarke generalized Jacobian $\partial {\rm Prox}_{\rho \|\cdot\|_*}$.

\begin{proposition}\label{prop:1}
Let $\rho>0$ and $Y\in \mathbb{R}^{m\times q}$ admit the singular value decomposition $$
Y = U[{\rm Diag}(\sigma_1,\dots,\sigma_m)\ 0] V^{\intercal} =  U[{\rm Diag}(\sigma_1,\dots,\sigma_m)\ 0] [V_1 \ V_2]^{\intercal},
$$
where $U\in \mathbb{R}^{m\times m}$, $V{= [V_1  \  V_2]}\in \mathbb{R}^{q\times q}$ are orthogonal matrices, $V_1 \in \mathbb{R}^{q\times m}$, $V_2 \in \mathbb{R}^{q\times (q-m)}$, and $\sigma_1\geq \sigma_2\geq \cdots\geq \sigma_m\geq 0$ are the singular values of $Y$. Then we have the following conclusions.
\begin{enumerate}[label=(\alph*)]	
\item The proximal mapping associated with $\rho \|\cdot\|_*$ at $Y$ can be computed as
\begin{align*}
{\rm Prox}_{\rho \|\cdot\|_*}(Y) = U[{\rm Diag}((\sigma_1 - \rho)_+,\dots,(\sigma_m - \rho)_+)\ 0] V^{\intercal}.
\end{align*}

\item The function ${\rm Prox}_{\rho \|\cdot\|_*}(\cdot)$ is strongly semismooth everywhere in $\mathbb{R}^{m\times q}$ with respect to $\partial {\rm Prox}_{\rho \|\cdot\|_*}(\cdot)$.

\item Define the index sets
\begin{align*}
&\alpha = \{1,\cdots,m\},\quad \gamma = \{m+1,\cdots,2m\},\quad \beta = \{2m+1,\cdots,m+q \},\\
&\alpha_1 = \{i\in [m] \mid \sigma_i \!>\! \rho\},\quad \alpha_2 = \{i\in [m] \mid \sigma_i \!=\! \rho\},\quad
\alpha_3 = \{i\in [m] \mid \sigma_i \!<\! \rho\}.
\end{align*}
We can construct one element $\mathcal{W}\in \partial {\rm Prox}_{\rho \|\cdot\|_*}(Y)$, where  $\mathcal{W}:\mathbb{R}^{m\times q}\rightarrow \mathbb{R}^{m\times q}$ is defined as
\begin{align*}
\mathcal{W}(H) \!=\! U\left[ \left( \Gamma_{\alpha \alpha} \!\odot\! \left( \frac{H_1\!\!+\!\!H_1^{\intercal}}{2}\right) \!+\! \Gamma_{\alpha \gamma}\! \odot \!\left( \frac{H_1\!\!-\!\!H_1^{\intercal}}{2}\right) \right) V_1^{\intercal} \!+\! \left( \Gamma_{\alpha \beta}\!\odot\! H_2\right)V_2^{\intercal} \right],
\end{align*}
for $H\in \mathbb{R}^{m\times q}$. Here $H_1 = U^{\intercal} H V_1$, $H_2 = U^{\intercal} H V_2$,  $\Gamma_{\alpha \alpha}\in \mathbb{S}^{m}$, $\Gamma_{\alpha \gamma}\in \mathbb{S}^{m}$, $\Gamma_{\alpha \beta}\in \mathbb{R}^{m\times (q-m)}$ are defined as
\begin{align*}
\Gamma_{\alpha \alpha}^0  \!\!=\!\! \begin{pmatrix}
1_{\alpha_1 \alpha_1} & \!\!\!\! 1_{\alpha_1 \alpha_2} & \!\!\!\!\tau_{\alpha_1 \alpha_3} \\
1_{\alpha_2 \alpha_1} & 0 & 0\\
\tau_{\alpha_1 \alpha_3}^{\intercal} & 0 & 0
\end{pmatrix}, \; \Gamma_{\alpha \gamma} \!\!=\!\! \begin{pmatrix}
\omega_{\alpha_1\alpha_1} & \!\!\!\!\omega_{\alpha_1\alpha_2} & \!\!\!\!\omega_{\alpha_1\alpha_3}\\
\omega_{\alpha_1\alpha_2}^{\intercal}  & 0 & 0\\
\omega_{\alpha_1\alpha_3}^{\intercal}  & 0 & 0
\end{pmatrix},\;
\Gamma_{\alpha \beta} \!\!=\!\! \begin{pmatrix}
\mu_{\alpha_1 \bar{\beta}} \\
0
\end{pmatrix},
\end{align*}
with $\tau_{\alpha_1 \alpha_3}\in \mathbb{R}^{|\alpha_1|\times |\alpha_3|}$, $\omega_{\alpha_1\alpha}\in \mathbb{R}^{|\alpha_1|\times m}$, $\mu_{\alpha_1 \bar{\beta}}\in \mathbb{R}^{|\alpha_1|\times (q-m)}$ defined as
\begin{align*}
&\tau_{ij} = \frac{\sigma_i - \rho}{\sigma_i -\sigma_j}, \quad \mbox{for } i \in \alpha_1,\ j\in \alpha_3,\\
&\omega_{ij} = \frac{\sigma_i-\rho + (\sigma_j-\rho)_{+}}{\sigma_i+\sigma_j},\quad \mbox{for } i \in \alpha_1,\ j\in \alpha,\\
& \mu_{ij} = \frac{\sigma_i-\rho}{\sigma_i},\quad \mbox{for } i\in \alpha_1,\ j\in \bar{\beta} :=  \{1,\cdots,q-m\}.
\end{align*}
\end{enumerate}
\end{proposition}
\begin{proof}
(a) follows from \citep{cai2010singular,ma2011fixed}.  (b) follows from \citep[Theorem 2.1]{jiang2014partial}.
(c)  follows from \citep[Lemma~2.3.6, Proposition~2.3.7]{yang2009study} and the fact that $$\partial_B {\rm Prox}_{\rho \|\cdot\|_*}(Y)\subseteq \partial {\rm Prox}_{\rho \|\cdot\|_*}(Y) = {\rm conv}(\partial_B {\rm Prox}_{\rho \|\cdot\|_*}(Y))$$ for any $Y\in \mathbb{R}^{m\times q}$.
\end{proof}
In addition to the above formulae, we need the following standard results for constructing one element in the set \eqref{eq: jacobian}:
\begin{align*}
&{\rm Prox}_{\sigma_k h/\nu} (u)   = \frac{u + \frac{\sigma_k}{\nu} y}{1+ \frac{\sigma_k}{\nu}},\quad \nabla{\rm Prox}_{\sigma_k h/\nu} (u) = \frac{1}{1+ \frac{\sigma_k}{\nu}}I_n, \\
&({\rm Prox}_{\sigma_k \psi} (\gamma))_i = {\rm sgn}(\gamma_i) (|\gamma_i| - \sigma_k \lambda)_+,\ i\in [p],  \\
&\partial {\rm Prox}_{\sigma_k \psi}(\gamma) =
\left\{ {\rm Diag}(w)\middle \vert \begin{aligned}
&w_i = 0 && \mbox{if } |\gamma_i| <\sigma_k \lambda, \\
&w_i \in [0,1] && \mbox{if } |\gamma_i| = \sigma_k \lambda, \\
&w_i = 1 && \mbox{if } |\gamma_i| > \sigma_k \lambda,
\end{aligned} \ i\in [p]
\right\}.
\end{align*}
Therefore, with one element $\mathcal{W}\in \partial {\rm Prox}_{\sigma_k\rho \|\cdot\|_*}(B^k - \sigma_k {\rm mat}( X^{\intercal}\xi))$ constructed via Proposition~\ref{prop:1}, we can construct one element {$\mathcal{H}:\mathbb{R}^n \rightarrow \mathbb{R}^n$ in the set $\hat{\partial }^2\Phi_k(\xi)$ such that for each $u\in \mathbb{R}^n$:
\begin{align*}
\mathcal{H}u  := & -\sigma_k X {\rm vec}( \mathcal{W} ( {\rm mat}(X^{\intercal}u))) -\sigma_k Z D Z^{\intercal}u- \frac{\sigma_k}{\nu + \sigma_k} u \\
= & -\sigma_k X {\rm vec}( \mathcal{W} ( {\rm mat}(X^{\intercal}u))) -\sigma_k Z_{\mathcal{A}} Z^{\intercal}_{\mathcal{A}}u- \frac{\sigma_k}{\nu + \sigma_k} u,
\end{align*}}
where $D\in \mathbb{S}^p$ is a diagonal matrix with its $(i,i)$-th element being $1$ if $i\in \mathcal{A} :=\{i \in [p]\mid |(\gamma^k -\sigma_k Z^{\intercal}\xi)_i| > \sigma_k \lambda \}$ and $0$ otherwise, and $Z_{\mathcal{A}}$ is the matrix consisting of the columns
of $Z$ indexed by $\mathcal{A}$. This reflects the second-order sparsity.

\section{Numerical Experiments}\label{sec:numerical}

In this section, we conduct extensive numerical experiments with two objectives. First, in Section~\ref{sec:modelcomparision}, we assess and compare the empirical performances, in terms of prediction and estimation accuracy, of various regularized models in the form of \eqref{eq: p} that employ different penalty functions $\phi$ and $\psi$.
This also shows the practical and general applicability of the PPDNA, accommodating  penalty functions of various forms.
Second, in Section~\ref{sec:efficiency}, we evaluate and contrast the efficiency of different optimization algorithms in terms of computational time, which aims to highlight the superior efficiency of the PPDNA among other state-of-the-art algorithms. {We take the squared loss function in \eqref{eq: p}.} All experiments were conducted using MATLAB 2019b on a computer with an i7-860, 2.80 GHz CPU, and 16 GB of RAM.

\subsection{Comparison of Empirical Performance}\label{sec:modelcomparision}

We evaluate the empirical performances of various estimators obtained from \eqref{eq: p}, each corresponding to a specific combination of penalty functions $\phi(B)$ and $\psi(\gamma)$. These penalties are as follows:
(i)  $\rho \|\operatorname{vec}(B)\|_1 + \lambda \|\gamma\|_1$, referred to as VML,  where the estimator is obtained by applying the lasso penalty after vectorizing the matrix variable;
(ii) $\rho \|B\|_* + \lambda \|\gamma\|_1$, referred to as NL as it includes the nuclear norm and lasso penalty;
(iii) $ \rho \|B\|_* + \lambda \|\gamma\|_1 + \lambda' \sum_{i=1}^{p-1}|\gamma_i -\gamma_{i+1}|$, referred to as NFL as it includes the nuclear norm and fused lasso penalty;
(iv) $\rho \|B\|_*+  \lambda \|\gamma\|_1 + \lambda' \sum_{l=1}^{g}\sqrt{w_l}\|\gamma_{G_l}\|_2$, referred to as NSGL as it includes the nuclear norm and sparse group lasso penalty, where $w_l>0$ is the weight for the $l$-th group,  $G_l$'s denote the group index sets forming a partition of $[p]$, and $\gamma_{G_l}$ is the subvector of $\gamma$ restricted to $G_l$.

In Section~\ref{sec:2d}, we explore on various two-dimensional geometrically shaped true matrix coefficient $B$ following \citet[Section~5.1]{zhou2014regularized}. Additionally, in Section~\ref{sec:model_syntheticdata}, we simulate on true matrix coefficient $B$ with varying ranks and non-sparsity levels (the percentage of nonzero elements) following \citet[Section~5.2]{zhou2014regularized}. As a given information for NSGL estimators, we construct the group structure simply by dividing the index set $[p]$ into adjacent groups $G_1,\dots,G_g$ with an approximately equal group size of around $p/g$.
The true vector coefficient $\gamma$ is generated using one of the following three schemes, each designed to exhibit specific characteristics:
\begin{enumerate}[itemsep=1pt, topsep=2pt, partopsep=0pt, parsep=0pt]
    \item[(S1)] Sparsity: we set $g=10$ and randomly choose an element in each group of $\gamma$ to be $5$, while the rest are $0$;
    \item[(S2)] Local constancy: we set $g=10$ and assign the first ten elements in each group of $\gamma$ to be $1$, with the remaining elements {being} $0$;
    \item[(S3)] Group structure: we set $g=20$. For the first ten groups of $\gamma$,  we alternate the first ten elements in each group between $1$ and $-1$, that is, $(1,-1,1,-1,\dots, 1, -1)$, with the remaining elements {being} $0$.
\end{enumerate}

We choose the tuning parameters to be
\begin{align}\label{eq: regset}
    \rho=\alpha_1\left\|{\rm mat}(X^{\intercal} y)\right\|_2,
    \qquad \lambda = \lambda' = \alpha_2\left\|Z^{\intercal} y\right\|_{\infty},
\end{align}
for NL, NFL and NSGL, and
\begin{equation*}
    \rho = \alpha_3\|X^{\intercal}y\|_{\infty}, \qquad \lambda=\alpha_4\left\|Z^{\intercal} y\right\|_{\infty},
\end{equation*}
for VML. Here, $0<\alpha_1, \alpha_2, \alpha_3, \alpha_4<1$. 
{More discussions on the selection of $\lambda$ and $\lambda'$ are given in Appendix~\ref{appendix: lambda}. On a training set of size $300$,} we search a large grid of values of $\alpha_1, \alpha_2, \alpha_3, \alpha_4$ in the range of $10^{-3}$ to $1$ with 20 equally divided grid points on the {$\log_{10}$} scale, and then select the parameters yielding the best predicting root-mean-square error (RMSE) on a validation data with a large sample size { of $3000$. Lastly, a testing data set of size $300$ serves to evaluate the chosen model's generalization ability on unseen data.  Similar procedure is also adopted by \citet[Section 5.1]{Li2020linearized}.}

To compare the empirical performances of different estimators, we consistently apply our proposed PPDNA for solving problem \eqref{eq: p} and derive the corresponding estimators. Specifically, we initialize the algorithm with {zeros} and terminate it when either the KKT residual is sufficiently small
\begin{equation}\label{kkt:stop6}
\eta_{\rm KKT} :=\max \{R_p, R_d, R_c\}< 10^{-6},
\end{equation}
or the number of iterations reaches $100$, where the formulae for $R_p$, $R_d$ and $R_c$ are in \eqref{kkt-res}.

We assess the empirical performance of our method from two aspects: parameter estimation and prediction {accuracy} based on $100$ repetitions. For the former, we compute the RMSE of $B$ and $\gamma$ {based on the estimated $(B^*,\gamma^*)$,} denoted as Error-$B$ and Error-$\gamma$. For the latter, we use {the testing data set} to evaluate the prediction error measured by the RMSE of the response $y$, denoted as RMSE-$y$.  Specifically,
$$
\text{RMSE-}y=\frac{ { \|y_{\rm test}-y^*_{\rm test}\|} }{\sqrt{n}},\quad\text{Error-}B=\frac{ { \|B-B^*\|_F}}{\sqrt{mq}}, \quad\text{Error-}\gamma=\frac{\|\gamma-\gamma^*\|}{\sqrt{p}},
$$
where $B$ and $\gamma$ are the true matrix and vector coefficients, respectively, $B^*$ and $\gamma^*$ are the corresponding estimated values, and $y^*_{\rm test}= X_{\rm test}{\rm vec}(B^*)+{Z_{\rm test}}\gamma^*$ with $\{(y_{\rm test},X_{\rm test},{Z_{\rm test})\}}$ being the testing data set.

\subsubsection{Two Dimensional Shapes}\label{sec:2d}

We set the problem dimensions as follows: $m=64$, $q=64$, and $p=1000$, with a sample size of $n=300$. The true matrix coefficient $B\in\{0,1\}^{m\times q}$ is binary, and its true signal region forms a two dimensional shape (square, triangle, circle, heart, {see Appendix \ref{sec:2d_appendix}}). The true vector coefficient $\gamma \in \mathbb{R}^p$ is generated using schemes (S1), (S2), and (S3). To generate the data, we randomly sample $\{ (X_i,z_i) \in\mathbb{R}^{m\times q} \times \mathbb{R}^p, i\in [n] \}$ such that each entry follows a standard normal distribution. The response $y_i\in \mathbb{R}$ also follows a normal distribution with mean $\langle X_i,B \rangle + \langle z_i,\gamma \rangle $ and variance $1$.

\begin{table}[!ht]
\centering
\caption{Comparison of the four estimators VML, NL, NFL, and NSGL {of} different true coefficients $B$ and $\gamma$, based on $100$ replications. Standard deviations are provided in parentheses below. The best value is highlighted in bold.}
{
\setlength{\tabcolsep}{1pt}
\begin{adjustbox}{width=0.99\linewidth}
\begin{tabular}{cccccccccccccccccc}
\hline
 &   & \multicolumn{4}{c}{RMSE-$y$} & &\multicolumn{4}{c}{Error-$B$}& & \multicolumn{4}{c}{Error-$\gamma$} \\
\cmidrule(r){3-6}  \cmidrule(r){8-11} \cmidrule(r){13-16}
$B$ & $\gamma$ & VML & NL & NFL & NSGL &  & VML & NL & NFL & NSGL&  & VML & NL & NFL & NSGL\\
\midrule
Square & S1 & 16.81 & \textbf{6.58} & 7.98 & 17.98 & & 0.23 & \textbf{0.09} & 0.11 & 0.19 & & 0.24 & \textbf{0.09} & 0.13 & 0.42 \\
       &    & (0.81) & (1.33) & (1.47) & (0.84) & & (0.00) & (0.02) & (0.02) & (0.01) & & (0.03) & (0.03) & (0.03) & (0.01) \\
       & S2 & 18.00 & 14.35 & \textbf{10.91} & 14.30 & & 0.23 & 0.16 & \textbf{0.14} & 0.17 & & 0.32 & 0.31 & \textbf{0.21} & 0.30 \\
       &    & (0.69) & (0.80) & (1.20) & (0.82) & & (0.00) & (0.01) & (0.02) & (0.01) & & (0.00) & (0.01) & (0.02) & (0.01) \\
       & S3 & 17.92 & 14.47 & 14.60 & \textbf{13.74} & & 0.23 & 0.17 & 0.17 & \textbf{0.16} & & 0.32 & 0.31 & 0.32 & \textbf{0.29} \\
       &    & (0.73) & (0.81) & (0.75) & (0.77) & & (0.00) & (0.01) & (0.01) & (0.01) & & (0.00) & (0.01) & (0.00) & (0.01) \\
\addlinespace
Triangle & S1 & 12.72 & \textbf{9.60} & 10.34 & 16.76 & & 0.18 & \textbf{0.13} & 0.14 & 0.17 & & 0.19 & \textbf{0.14} & 0.17 & 0.41 \\
         &    & (0.64) & (0.73) & (0.78) & (0.64) & & (0.00) & (0.01) & (0.01) & (0.01) & & (0.02) & (0.02) & (0.02) & (0.01) \\
         & S2 & 15.09 & 13.97 & \textbf{11.78} & 13.99 & & 0.18 & 0.16 & \textbf{0.15} & 0.16 & & 0.31 & 0.31 & \textbf{0.23} & 0.30 \\
         &    & (0.59) & (0.62) & (0.74) & (0.68) & & (0.00) & (0.01) & (0.01) & (0.01) & & (0.01) & (0.01) & (0.02) & (0.00) \\
         & S3 & 15.06 & 13.99 & 14.20 & \textbf{13.87} & & 0.18 & \textbf{0.16} & \textbf{0.16} & \textbf{0.16} & & 0.31 & 0.31 & 0.32 & \textbf{0.30} \\
         &    & (0.56) & (0.67) & (0.64) & (0.73) & & (0.00) & (0.01) & (0.01) & (0.01) & & (0.00) & (0.00) & (0.00) & (0.01) \\
\addlinespace
Circle & S1 & 11.12 & \textbf{7.84} & 8.60 & 16.19 & & 0.15 & \textbf{0.11} & \textbf{0.11} & 0.16 & & 0.16 & \textbf{0.11} & 0.14 & 0.40 \\
       &    & (0.56) & (0.69) & (0.83) & (0.66) & & (0.00) & (0.01) & (0.01) & (0.00) & & (0.02) & (0.02) & (0.02) & (0.01) \\
       & S2 & 14.06 & 13.29 & \textbf{9.94} & 12.86 & & 0.16 & 0.14 & \textbf{0.12} & 0.14 & & 0.31 & 0.30 & \textbf{0.19} & 0.29 \\
       &    & (0.57) & (0.69) & (0.79) & (0.66) & & (0.00) & (0.01) & (0.01) & (0.01) & & (0.00) & (0.01) & (0.02) & (0.01) \\
       & S3 & 14.06 & 13.52 & 13.36 & \textbf{12.53} & & 0.16 & 0.15 & \textbf{0.14} & \textbf{0.14} & & 0.31 & 0.31 & 0.32 & \textbf{0.28} \\
       &    & (0.55) & (0.64) & (0.59) & (0.62) & & (0.00) & (0.01) & (0.01) & (0.01) & & (0.01) & (0.01) & (0.00) & (0.01) \\
\addlinespace
Heart & S1 & 12.59 & \textbf{9.21} & 9.78 & 17.02 & & 0.17 & \textbf{0.13} & \textbf{0.13} & 0.17 & & 0.18 & \textbf{0.14} & 0.16 & 0.42 \\
      &    & (0.65) & (0.81) & (0.95) & (0.88) & & (0.00) & (0.01) & (0.01) & (0.01) & & (0.03) & (0.02) & (0.03) & (0.01) \\
      & S2 & 14.97 & 14.01 & \textbf{11.25} & 13.50 & & 0.18 & 0.15 & \textbf{0.14} & 0.15 & & 0.31 & 0.31 & \textbf{0.22} & 0.30 \\
      &    & (0.58) & (0.73) & (0.84) & (0.65) & & (0.00) & (0.01) & (0.01) & (0.01) & & (0.01) & (0.01) & (0.02) & (0.00) \\
      & S3 & 15.10 & 13.78 & 14.07 & \textbf{13.51} & & 0.18 & \textbf{0.15} & \textbf{0.15} & \textbf{0.15} & & 0.32 & 0.31 & 0.32 & \textbf{0.29} \\
      &    & (0.66) & (0.73) & (0.71) & (0.73) & & (0.00) & (0.01) & (0.01) & (0.01) & & (0.00) & (0.01) & (0.00) & (0.01) \\
\hline
\end{tabular}
\end{adjustbox}
}
\label{table:2d}
\end{table}

In Table~\ref{table:2d}, we present the RMSE-$y$, Error-$B$, and Error-$\gamma$ of the four estimators VML, NL, NFL, and NSGL. These estimators are evaluated under various scenarios where the true coefficients $B$ and $\gamma$ are generated differently. From Table~\ref{table:2d}, we observe that the VML estimator consistently fails to achieve the best performance. This observation suggests that the nuclear norm penalty function is more suitable than the {element-wise} $\ell_1$ norm penalty function for $B$, especially for certain shapes in this context. Moreover, when the true vector coefficient is generated using scheme (S1), the NL estimator demonstrates the best performances in {all} cases. Similarly, for scheme (S2) (and (S3)), the NFL (and NSGL) estimator generally exhibits {the} best performance. We may conclude that the empirical performance of the regularized regression model \eqref{eq: p} heavily depends on the selection of appropriate penalty functions to capture the characteristics of the true coefficients. Our results also demonstrate the versatility of our PPDNA for solving \eqref{eq: p} with various penalties.

For illustrative purposes, we display in Figure~\ref{fig:2d} the true coefficients and the estimated coefficients by VML, NL, NFL, and NSGL for a square shaped $B$. We can see from the second column of Figure~\ref{fig:2d} that the VML estimator exhibits the least effective visual recovery. Notably, in the second
row of Figure~\ref{fig:2d}, the NFL estimator demonstrates the best recovery of truth, while in the third
row the NSGL estimator achieves the best recovery.

\begin{figure}[!ht]
\includegraphics[width=1\linewidth]{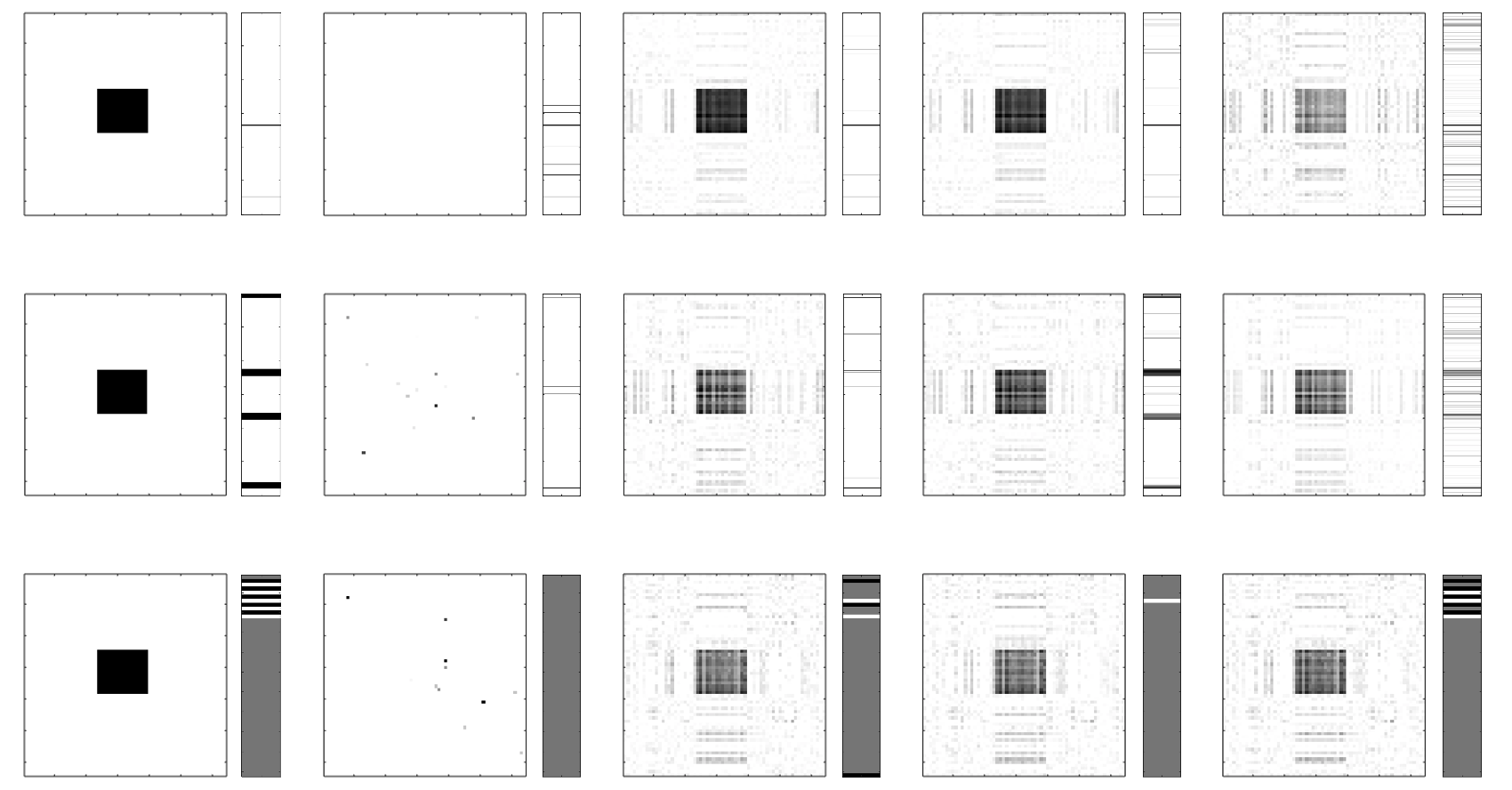} 
 \text{\space}~~~~~~Truth~~~~~~~~~~~~~~~~~VML~~~~~~~~~~~~~~~~~~~NL~~~~~~~~~~~~~~~~~~~~NFL~~~~~~~~~~~~~~~~~NSGL~~~
\caption{Comparison of the four estimators VML, NL, NFL, and NSGL under square shaped $B$ and different $\gamma$ generating schemes  (S1), (S2), and (S3) from top to bottom rows. In each subgraph, the left square depicts the matrix coefficients, while the right rectangle depicts the vector coefficients. For clarity, segment from 100 to 400 of estimated vector coefficients {is} displayed for (S1) and (S2), while for (S3), the segment {is} from 210 to 260 (5th group). For every estimated matrix coefficients, we set the color limit to be 0 and 1, that is, entries with non-positive values are mapped to white, entries with values greater than 1 are mapped to black, and entries with values between 0 and 1 are uniformly mapped to a color scale ranging from white to black. For estimated vector coefficients, we set the color limit to be 0 and 1 for (S1) and (S2), while for (S3), we map entries whose absolute values are smaller than 0.2 to gray, and the entries that are greater than $0.2$ to black, and those less than $-0.2$ to white to highlight the group structure.}
\label{fig:2d}
\end{figure}

\subsubsection{Synthetic Data}\label{sec:model_syntheticdata}
We set the problem dimensions as follows: $m = 50$, $q = 50,$ and $p=1000$, with a sample size of $n=300$. We generate the true matrix coefficient $B$ as the product of two matrices $B = B_1 B_2^{\intercal}$, where $ B_1 \in\mathbb{R}^{m \times r}$, $ B_2 \in{\mathbb{R}^{q \times r}}$, and $r$ controls the rank of $B$. Each entry of $B_1$ and $B_2$ is independently generated from a Bernoulli distribution with probability of one equal to $\sqrt{1-(1-s)^{1/r}}$. We can deduce that the non-sparsity level (the percentage of nonzero elements) of $B$ is expected to be $s$. We choose the rank $r = 1,2,4$, and the non-sparsity level $s=0.1,0.2,0.4$. The true vector coefficient $\gamma$ and samples $(y_i,X_i,z_i)$ are generated in the same manner as described in  Section~\ref{sec:2d}.

Tables~\ref{table:synthetic_new_data_sparse}, \ref{table:synthetic_new_data_fused}, and \ref{table:synthetic_new_data_group}   present the  RMSE-$y$, Error-$B$, and Error-$\gamma$ of the four estimators VML, NL, NFL, and NSGL evaluated under various scenarios. These scenarios involve different ranks $r$ and non-sparsity levels $s$ for the true coefficients $B$, while $\gamma$ is generated using schemes (S1), (S2), and (S3).
From Tables~\ref{table:synthetic_new_data_sparse} to \ref{table:synthetic_new_data_group}, several observations can be made {as follows.} First, the VML estimator has the highest RMSE-$y$, Error-$B$, Error-$\gamma$ in all scenarios, which implies the inferiority of the $\ell_1$ penalty $\|{\rm vec}(B)\|_1$ compared to the nuclear norm penalty $\|B\|_*$. Second, when $\gamma$ is generated using scheme (S1), the NL estimator consistently demonstrates the lowest error in the estimation of  $B$, $\gamma$ and in the prediction of $y$. Similarly, the NFL estimator has the lowest error in most cases in Table~\ref{table:synthetic_new_data_fused} where $\gamma$ is generated using scheme (S2), and the NSGL estimator generally achieves the lowest error in Table~\ref{table:synthetic_new_data_group} where $\gamma$ is generated using scheme (S3).
Again, we conclude that the choice of penalty functions is contingent on the specific characteristics of predictors. And our simulations reflect the versatility and generality of the PPDNA.

\begin{table}[!ht]
\centering
\caption{Comparison of the four estimators VML, NL, NFL, and NSGL under $B$ with different ranks $r$ and non-sparisity levels $s${, together with} $\gamma$ generating schemes (S1), based on $100$ replications.
}
{
\setlength{\tabcolsep}{1pt}
\begin{adjustbox}{width=0.99\linewidth}
\begin{tabular}{ccccccccccccccccc}
\hline
 &&  & \multicolumn{4}{c}{RMSE-$y$} & &\multicolumn{4}{c}{Error-$B$}&& \multicolumn{4}{c}{Error-$\gamma$} \\
\cmidrule(r){4-7}  \cmidrule(r){9-12} \cmidrule(r){14-17}
$s$ && $r$   & VML & NL & NFL & NSGL &  & VML & NL & NFL & NSGL&  & VML & NL & NFL & NSGL \\
\hline
0.1 && 1 & 17.799 & \textbf{3.478} & 4.067 & 17.210 && 0.317 & \textbf{0.058} & 0.067 & 0.221 && 0.253 & \textbf{0.052} & 0.066 & 0.417 \\
    &&   & (2.876) & (0.868) & (1.112) & (1.207) && (0.048) & (0.015) & (0.019) & (0.023) && (0.050) & (0.016) & (0.021) & (0.012) \\
 && 2 & 18.510 & \textbf{11.167} & 12.950 & 19.595 && 0.323 & \textbf{0.197} & 0.218 & 0.276 && 0.279 & \textbf{0.162} & 0.215 & 0.437 \\
    &&   & (3.165) & (1.885) & (2.218) & (1.753) && (0.053) & (0.032) & (0.036) & (0.034) && (0.059) & (0.036) & (0.046) & (0.019) \\
 && 4 & 18.838 & \textbf{15.686} & 16.387 & 20.605 && 0.328 & \textbf{0.274} & 0.279 & 0.302 && 0.285 & \textbf{0.234} & 0.267 & 0.440 \\
    &&   & (2.555) & (2.099) & (2.164) & (1.580) && (0.040) & (0.032) & (0.032) & (0.030) && (0.049) & (0.046) & (0.047) & (0.013) \\
    \addlinespace
0.2 && 1 & 24.939 & \textbf{3.555} & 4.787 & 18.813 && 0.450 & \textbf{0.060} & 0.080 & 0.265 && 0.348 & \textbf{0.051} & 0.078 & 0.426 \\
    &&   & (3.307) & (0.899) & (1.443) & (1.336) && (0.055) & (0.016) & (0.024) & (0.024) && (0.050) & (0.015) & (0.026) & (0.012) \\
 && 2 & 27.117 & \textbf{15.472} & 17.316 & 22.921 && 0.483 & \textbf{0.273} & 0.296 & 0.354 && 0.379 & \textbf{0.219} & 0.277 & 0.454 \\
    &&   & (3.524) & (2.097) & (2.295) & (1.642) && (0.064) & (0.037) & (0.039) & (0.034) && (0.057) & (0.041) & (0.047) & (0.014) \\
 && 4 & 27.792 & \textbf{21.971} & 23.150 & 25.708 && 0.496 & \textbf{0.390} & 0.398 & 0.415 && 0.393 & \textbf{0.317} & 0.369 & 0.471 \\
    &&   & (2.956) & (2.478) & (2.413) & (1.859) && (0.049) & (0.039) & (0.038) & (0.035) && (0.046) & (0.054) & (0.049) & (0.014) \\
    \addlinespace
0.4 && 1 & 35.279 & \textbf{3.903} & 5.714 & 20.764 && 0.640 & \textbf{0.066} & 0.096 & 0.307 && 0.457 & \textbf{0.055} & 0.092 & 0.440 \\
    &&   & (2.972) & (1.220) & (2.205) & (1.598) && (0.050) & (0.022) & (0.038) & (0.031) && (0.047) & (0.019) & (0.038) & (0.012) \\
 && 2 & 40.383 & \textbf{22.738} & 22.800 & 26.603 && 0.740 & \textbf{0.395} & 0.394 & 0.440 && 0.488 & \textbf{0.349} & 0.354 & 0.470 \\
    &&   & (4.272) & (2.760) & (2.719) & (1.954) && (0.075) & (0.046) & (0.045) & (0.037) && (0.052) & (0.053) & (0.053) & (0.013) \\
 && 4 & 43.102 & \textbf{30.262} & 31.635 & 31.792 && 0.809 & \textbf{0.546} & 0.559 & 0.558 && 0.483 & \textbf{0.411} & 0.471 & 0.484 \\
    &&   & (4.151) & (2.629) & (2.276) & (2.372) && (0.074) & (0.042) & (0.040) & (0.041) && (0.034) & (0.041) & (0.020) & (0.011) \\
\hline
\end{tabular}
\end{adjustbox}
}
\label{table:synthetic_new_data_sparse}
\end{table}

\begin{table}[!ht]
\centering
\caption{Comparison of the four estimators VML, NL, NFL, and NSGL under $B$ with different ranks $r$ and non-sparisity levels $s${, together with} $\gamma$ generating schemes (S2), based on $100$ replications.
}
{
\setlength{\tabcolsep}{1pt}
\begin{adjustbox}{width=0.99\linewidth}
\begin{tabular}{ccccccccccccccccc}
\hline
 &&  & \multicolumn{4}{c}{RMSE-$y$} & &\multicolumn{4}{c}{Error-$B$}&& \multicolumn{4}{c}{Error-$\gamma$} \\
\cmidrule(r){4-7}  \cmidrule(r){9-12} \cmidrule(r){14-17}
$s$ && $r$   & VML & NL & NFL & NSGL &  & VML & NL & NFL & NSGL&  & VML & NL & NFL & NSGL\\
\hline
0.1 && 1 & 18.489 & 14.018 & \textbf{7.982} & 13.221 && 0.310 & 0.189 & \textbf{0.121} & 0.184 && 0.315 & 0.326 & \textbf{0.160} & 0.299 \\
    &&   & (2.348) & (1.253) & (1.764) & (1.043) && (0.053) & (0.023) & (0.027) & (0.021) && (0.002) & (0.016) & (0.038) & (0.008) \\
 && 2 & 19.281 & 15.860 & \textbf{13.931} & 16.151 && 0.326 & 0.246 & \textbf{0.228} & 0.251 && 0.317 & 0.314 & \textbf{0.248} & 0.314 \\
    &&   & (2.404) & (1.373) & (1.954) & (1.577) && (0.052) & (0.030) & (0.033) & (0.033) && (0.006) & (0.008) & (0.028) & (0.011) \\
 && 4 & 19.335 & 17.490 & \textbf{16.530} & 17.383 && 0.330 & 0.286 & \textbf{0.281} & 0.285 && 0.318 & 0.319 & \textbf{0.274 }& 0.313 \\
    &&   & (1.919) & (1.505) & (1.627) & (1.426) && (0.043) & (0.032) & (0.031) & (0.031) && (0.007) & (0.007) & (0.014) & (0.002) \\
    \addlinespace
0.2 && 1 & 24.530 & 14.715 & \textbf{10.639} & 14.030 && 0.448 & 0.217 & \textbf{0.166} & 0.208 && 0.316 & 0.316 &\textbf{ 0.208} & 0.299 \\
    &&   & (2.551) & (1.004) & (2.056) & (1.095) && (0.053) & (0.020) & (0.034) & (0.021) && (0.000) & (0.001) & (0.038) & (0.005) \\
 && 2 & 26.568 & 18.947 & \textbf{17.431} & 18.446 && 0.488 & 0.316 & \textbf{0.301} & 0.312 && 0.316 & 0.328 & \textbf{0.277} & 0.309 \\
    &&   & (2.910) & (1.561) & (1.666) & (1.485) && (0.059) & (0.030) & (0.030) & (0.028) && (0.001) & (0.011) & (0.017) & (0.004) \\
 && 4 & 26.840 & 21.835 &\textbf{ 21.761} & 21.828 && 0.493 &\textbf{ 0.385} & \textbf{0.385} & \textbf{0.385} && 0.317 & 0.316 & \textbf{0.310} & 0.316 \\
    &&   & (2.557) & (1.751) & (1.843) & (1.752) && (0.055) & (0.035) & (0.036) & (0.035) && (0.003) & (0.000) & (0.020) & (0.001) \\
    \addlinespace
0.4 && 1 & 33.062 & 15.526 & \textbf{11.578} & 15.175 && 0.628 & 0.238 & \textbf{0.184} & 0.234 && 0.319 & 0.314 &\textbf{ 0.218} & 0.306 \\
    &&   & (2.592) & (1.165) & (2.086) & (1.083) && (0.044) & (0.024) & (0.035) & (0.022) && (0.004) & (0.007) & (0.035) & (0.004) \\
 && 2 & 38.925 & 21.139 & \textbf{20.426} & 21.380 && 0.750 & 0.373 & \textbf{0.364} & 0.377 && 0.326 & 0.316 & \textbf{0.297} & 0.323 \\
    &&   & (3.496) & (1.640) & (1.903) & (1.820) && (0.065) & (0.031) & (0.034) & (0.034) && (0.013) & (0.001) & (0.019) & (0.010) \\
 && 4 & 41.099 & 27.812 & \textbf{27.806} & \textbf{27.806} && 0.792 & \textbf{0.517} & \textbf{0.517} & \textbf{0.517} && \textbf{0.316} & \textbf{0.316} & \textbf{0.316} & \textbf{0.316} \\
    &&   & (4.182) & (2.018) & (2.013) & (2.025) && (0.081) & (0.039) & (0.039) & (0.039) && (0.000) & (0.000) & (0.001) & (0.001) \\
\hline
\end{tabular}
\end{adjustbox}
}
\label{table:synthetic_new_data_fused}
\end{table}
\begin{table}[!ht]
\centering
\caption{Comparison of the four estimators VML, NL, NFL, and NSGL under $B$ with different ranks $r$ and non-sparisity levels $s${, together with} $\gamma$ generating schemes  (S3), based on $100$ replications.
}
{
\setlength{\tabcolsep}{1pt}
\begin{adjustbox}{width=0.99\linewidth}
\begin{tabular}{ccccccccccccccccc}
\hline
 &  && \multicolumn{4}{c}{RMSE-$y$} & &\multicolumn{4}{c}{Error-$B$}&& \multicolumn{4}{c}{Error-$\gamma$} \\
\cmidrule(r){4-7}  \cmidrule(r){9-12} \cmidrule(r){14-17}
$s$ && $r$   & VML & NL & NFL & NSGL &  & VML & NL & NFL & NSGL&  & VML & NL & NFL & NSGL\\
\hline
0.1 && 1 & 18.195 & 13.311 & 13.618 & \textbf{12.499} && 0.304 & 0.182 & 0.186 & \textbf{0.174} && 0.316 & 0.306 & 0.315 & \textbf{0.281} \\
    &&   & (2.120) & (0.762) & (0.788) & (0.842) && (0.048) & (0.018) & (0.018) & (0.019) && (0.001) & (0.005) & (0.005) & (0.008) \\
 && 2 & 18.664 & 15.556 & 15.659 &\textbf{ 15.099} && 0.316 & 0.241 & 0.242 & \textbf{0.237} && 0.316 & 0.311 & 0.316 & \textbf{0.297} \\
   & &   & (2.017) & (1.187) & (1.182) & (1.271) && (0.042) & (0.024) & (0.024) & (0.025) && (0.001) & (0.005) & (0.000) & (0.009) \\
 & &4 & 19.700 & 17.838 & 17.790 & \textbf{17.516} && 0.335 & 0.291 & 0.290 & \textbf{0.289} && 0.317 & 0.315 & 0.316 & \textbf{0.306} \\
    &&   & (1.908) & (1.402) & (1.388) & (1.460) && (0.040) & (0.029) & (0.029) & (0.029) && (0.004) & (0.004) & (0.000) & (0.008) \\
        \addlinespace
0.2 && 1 & 24.181 & 14.593 & 14.638 & \textbf{13.747} && 0.440 & 0.215 & 0.213 & \textbf{0.204} && 0.318 & 0.310 & 0.315 & \textbf{0.291} \\
    &&   & (2.558) & (0.914) & (0.855) & (0.893) && (0.056) & (0.020) & (0.020) & (0.019) && (0.005) & (0.007) & (0.003) & (0.006) \\
 && 2 & 25.944 & 18.245 & 18.276 & \textbf{18.073} && 0.477 & 0.306 & 0.306 & \textbf{0.305} && 0.327 & 0.315 & 0.316 & \textbf{0.308} \\
    &&   & (2.899) & (1.420) & (1.430) & (1.464) && (0.059) & (0.030) & (0.030) & (0.030) && (0.011) & (0.003) & (0.001) & (0.005) \\
 && 4 & 27.183 & \textbf{22.197} & 22.205 & 22.559 && 0.507 & \textbf{0.396} & 0.397 & 0.401 && \textbf{0.316} & \textbf{0.316} & \textbf{0.316} & 0.332 \\
    &&   & (2.585) & (1.790) & (1.794) & (1.961) && (0.050) & (0.032) & (0.032) & (0.034) && (0.000) & (0.001) & (0.000) & (0.013) \\
        \addlinespace
0.4 && 1 & 33.195 & 15.414 & 15.583 & \textbf{14.717} && 0.631 & 0.237 & 0.239 & \textbf{0.227} && 0.319 & 0.312 & 0.316 & \textbf{0.293} \\
    &&   & (2.948) & (1.099) & (1.127) & (1.136) && (0.056) & (0.022) & (0.023) & (0.022) && (0.004) & (0.003) & (0.000) & (0.007) \\
 && 2 & 38.469 & 21.212 & 21.239 & \textbf{21.163} && 0.743 & 0.374 & 0.374 & \textbf{0.373} && 0.332 & 0.316 & 0.317 &\textbf{ 0.315} \\
    &&   & (3.454) & (1.843) & (1.841) & (1.835) && (0.062) & (0.035) & (0.035) & (0.035) && (0.015) & (0.000) & (0.002) & (0.002) \\
 && 4 & 42.049 & \textbf{28.281} & \textbf{28.281} & 28.313 && 0.817 & \textbf{0.528} & \textbf{0.528} & 0.529 && \textbf{0.316 }& \textbf{0.316} & \textbf{0.316} & \textbf{0.316} \\
    &&   & (3.872) & (2.282) & (2.282) & (2.288) && (0.073) & (0.039) & (0.039) & (0.039) && (0.000) & (0.000) & (0.000) & (0.001) \\
\hline
\end{tabular}
\end{adjustbox}
}
\label{table:synthetic_new_data_group}
\end{table}

\subsection{Comparison of Algorithm Efficiency}\label{sec:efficiency}
This subsection aims to reveal the relative performance and applicability of the PPDNA for solving \eqref{eq: p}, highlighting its potential efficiency over established methods such as {the ADMM and the Nesterov} algorithm. For illustrative purpose, we choose 
the penalty function to be $\phi(B) + \psi(\gamma) = \rho {\|B\|_*} + \lambda \|\gamma\|_1$, resulting  NL estimators. Other forms of $\psi$ yield quite similar results in terms of {the} performance of algorithm efficiency.

We compare the efficiency of our PPDNA with two state-of-the-art algorithms: the {Nesterov} algorithm implemented by \citet{zhou2014regularized} and the alternating direction method of multipliers (ADMM), of which the details are given in Appendix~\ref{sec:admm}.  It is worth noting that the {Nesterov} algorithm presented by \citet{zhou2014regularized} was primarily designed for estimating the low rank matrix coefficient $B$ without regularization on the vector coefficient $\gamma$ ($\lambda = 0$). In our experiments, we implemented the {Nesterov} algorithm  \citep[Algorithm~1]{zhou2014regularized}  from  \url{https://hua-zhou.github.io/SparseReg} and \url{https://hua-zhou.github.io/TensorReg}. We implemented the ADMM {by} ourselves, see Appendix~\ref{sec:admm} for more details. We note that there is an ADMM type method \citep[Algorithm~1]{Li2020linearized} for obtaining the NFL estimator, while their codes are not available.

For each instance, we run the PPDNA to a high accuracy, terminating {after a maximum of 50 iterations}
or when $\eta_{\rm KKT} < 10^{-10}$ (see \eqref{kkt:stop6}), and use the resulting solution as a benchmark. We denote the benchmark objective function value of \eqref{eq: p} as ${\rm obj}^*$. For the objective function value of one method, denoted as ${\rm obj}$, we compute the relative objective function value
\begin{equation}\label{eq:Robj}
\begin{array}{l}
    R_{\rm obj} := ({{\rm obj} - {\rm obj}^*})/({1 + {|{\rm obj}^*|}}).
\end{array}
\end{equation}
We initialize each algorithm with zero and set a maximum computational time of one hour. For the PPDNA, we terminate execution when either $R_{\rm obj}<10^{-10}$ or the number of iterations reaches $50$. While for the ADMM and {the Nesterov} algorithm, termination occurs when either $R_{\rm obj}<10^{-10}$ or the number of iterations reaches $500,000$.

As mentioned before, the {Nesterov} algorithm \citep[Algorithm~1]{zhou2014regularized} does not support regularization on the vector coefficient $\gamma$ ({that is, it is only capable in the case when} $\lambda = 0$). In order to include {the} {Nesterov} algorithm \citep[Algorithm~1]{zhou2014regularized} into comparison, we {first} use the regularization setting:
\begin{align}\label{eq:set1}
    \rho=\alpha\|{\rm mat}(X^{\intercal} y)\|_{2} \mbox{ with } 0<\alpha<1,\; \lambda = 0,
\end{align}
and compare the three algorithms PPDNA, ADMM, and Nesterov algorithm under this setting. Additionally, we {also} compare PPDNA and ADMM under the
setting as in \eqref{eq: regset}:
\begin{align}\label{eq:set2}
   \rho=\alpha_1\|{\rm mat}(X^{\intercal} y)\|_{2}, \quad \lambda=\alpha_2\|Z^{\intercal} y\|_{\infty},\mbox{ with } \alpha_1,\alpha_2\in (0,1).
\end{align}

\begin{figure}[!ht]
    \centering
     \includegraphics[width=0.24\linewidth]{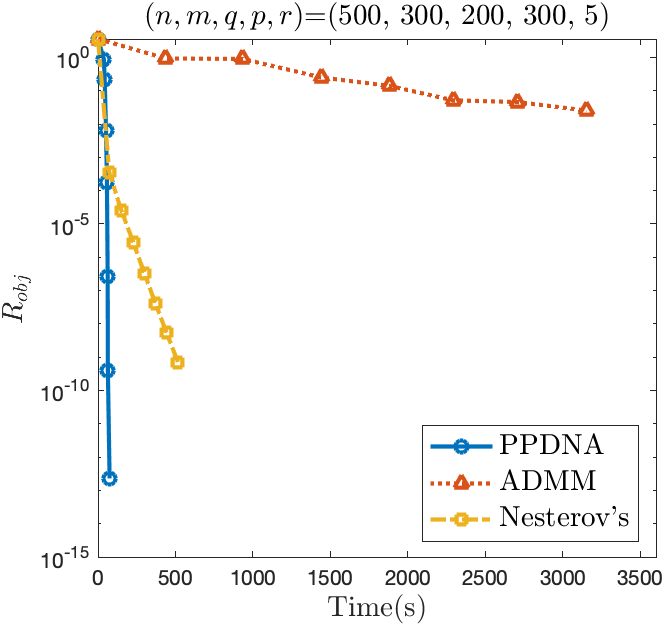}
        \includegraphics[width=0.24\linewidth]{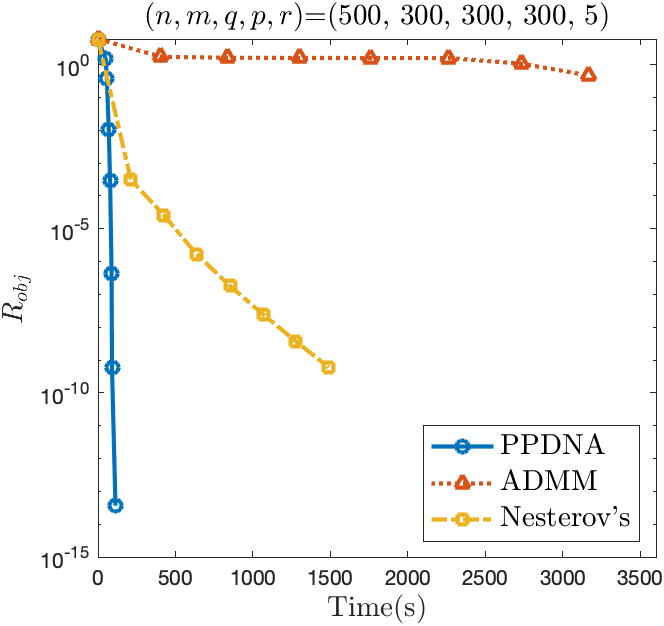}
        \includegraphics[width=0.24\linewidth]{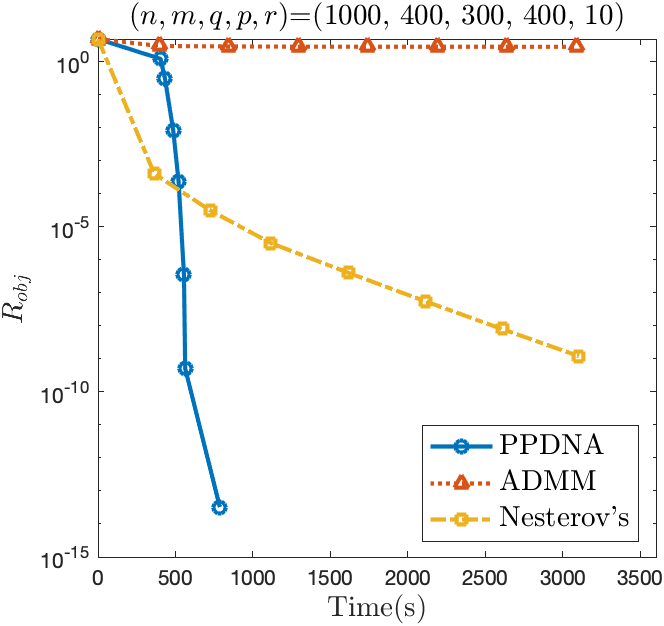}
        \includegraphics[width=0.24\linewidth]{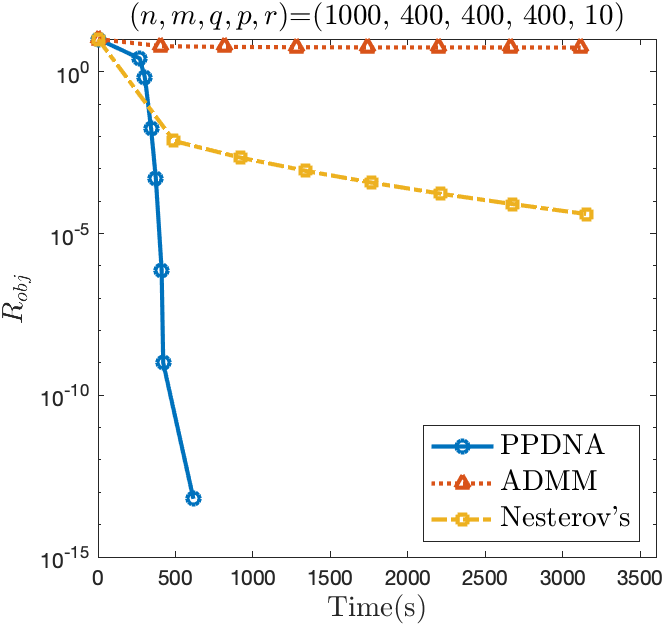}
    \caption{Time vs. $R_{\rm obj}$ (see \eqref{eq:Robj}) of PPDNA, ADMM, and Nesterov algorithm on synthetic data. The penalty parameters are taken from Table~\ref{table:new_format_update} with {asterisks} (*).
    }\label{fig:efficiency_synthetic_noVec}
\end{figure}

\subsubsection{Synthetic Data}
In this experiment, we adopt the same data generation strategy as in Section~\ref{sec:model_syntheticdata}, with some minor adjustments. We set $s=0.1$ and randomly generate a true $\gamma$ with a non-sparsity level of 0.01, where all non-zero entries are ones.
We conduct experiments on synthetic data across various problem dimensions and sample sizes. Additionally, we explore different values of $(\rho,\lambda)$, including pairs capable of recovering the true rank of $B$ and the true sparsity level of $\gamma$. The results are shown in Tables~\ref{table:new_format_update} and \ref{table:synthetic_efficiency_vec}. Additionally, we plot
Figures~\ref{fig:efficiency_synthetic_noVec} and \ref{fig:efficiency_synthetic_vec}, corresponding to the {parameters} with an asterisk from Tables~\ref{table:new_format_update} and \ref{table:synthetic_efficiency_vec}.

\begin{table}[!h]
  \centering
  \caption{Time comparison of PPDNA, ADMM, and Nesterov algorithm with penalty parameters in \eqref{eq:set1} on synthetic data. Here $\hat{r}$ represents the rank of the estimated matrix coefficient $B^*$ by PPDNA, and $R_{\rm obj}$ is given by  \eqref{eq:Robj}.  ``P'' stands for PPDNA; ``A'' for ADMM; ``N'' for Nesterov algorithm. The penalty parameters marked with asterisks (*) are used for plotting Figure~\ref{fig:efficiency_synthetic_noVec}, as they have been found to recover the true rank of $B$.
  }
{
\setlength{\tabcolsep}{3pt}
\begin{adjustbox}{width=0.99\linewidth}
  \begin{tabular}{clcccccccccc}
\hline
& && \multicolumn{3}{c}{Iterations} & \multicolumn{3}{c}{$R_{\rm obj}$} & \multicolumn{3}{c}{Time} \\
\cmidrule(r){4-6}\cmidrule(r){7-9}\cmidrule(r){10-12}
Data &$\rho$&$\hat{r}$ & P & A & N & P & A& N & P & A&N \\
\midrule
& 1.04e+1 & 6 & 13 & 925 & 43797 & 3.01e-16 & 9.89e-11 & 4.76e-02 & 0:01:02 & 0:02:32 & 1:00:00 \\
$n=500$& 1.04e+2 & 6 & 10 & 16199 & 38902 & 3.23e-11 & 9.02e-11 & 3.71e-03 & 0:00:42 & 0:34:39 & 1:00:00 \\
$m=300$& 1.04e+3 & 6 & 7 & 19362 & 42416 & 4.36e-12 & 3.94e-05 & 7.71e-08 & 0:00:49 & 1:00:00 & 1:00:00 \\
$q=200$& 6.91e+3 & 6 & 7 & 17783 & 11269 & 3.35e-14 & 2.82e-01 & 9.99e-11 & 0:01:28 & 1:00:00 & 0:13:27 \\
$p=300$& 1.04e+4* & 5 & 7 & 17354 & 7084 & 2.30e-13 & 1.48e-01 & 9.98e-11 & 0:01:15 & 1:00:00 & 0:09:47 \\
$r=5$& 1.38e+4 & 4 & 7 & 11228 & 4739 & 8.92e-14 & 7.71e-02 & 1.00e-11 & 0:01:56 & 1:00:00 & 0:07:29 \\
& 2.07e+4 & 4 & 7 & 14803 & 3031 & 1.33e-14 & 4.65e-02 & 9.99e-11 & 0:02:00 & 1:00:00 & 0:04:46 \\
\midrule
& 1.03e+1 & 6 & 12 & 1059 & 30816 & 5.97e-11 & 9.94e-11 & 9.98e-02 & 0:01:15 & 0:04:13 & 1:00:00 \\
$n=500$& 1.03e+2 & 6 & 10 & 18093 & 29772 & 2.85e-11 & 1.43e-08 & 1.79e-02 & 0:01:26 & 1:00:00 & 1:00:00 \\
$m=300$& 1.03e+3 & 6 & 7 & 12541 & 34291 & 1.93e-11 & 1.19e+00 & 8.16e-05 & 0:01:04 & 1:00:00 & 1:00:00 \\
$n=300$& 5.17e+3 & 6 & 7 & 11643 & 38344 & 6.42e-14 & 2.62e-01 & 3.40e-10 & 0:02:09 & 1:00:00 & 1:00:00 \\
$p=300$& 1.03e+4* & 5 & 7 & 11784 & 16496 & 3.65e-14 & 2.24e-01 & 1.00e-10 & 0:01:53 & 1:00:00 & 0:28:22 \\
$r=5$& 2.07e+4 & 5 & 7 & 10851 & 7818 & 4.38e-13 & 6.72e-01 & 9.99e-11 & 0:01:53 & 1:00:00 & 0:12:20 \\
& 3.10e+4 & 3 & 7 & 10700 & 2494 & 3.87e-13 & 1.44e-01 & 9.96e-11 & 0:01:46 & 1:00:00 & 0:04:00 \\
\midrule
& 1.51e+1 & 11 & 12 & 1019 & 15443 & 7.76e-13 & 9.76e-11 & 1.42e-01 & 0:02:08 & 0:08:24 & 1:00:00 \\
$n=1000$& 1.51e+2 & 11 & 8 & 5010 & 13846 & 2.54e-11 & 9.95e-11 & 3.20e-02 & 0:02:19 & 0:59:10 & 1:00:00 \\
$m=400$& 1.51e+3 & 11 & 7 & 4137 & 12688 & 1.14e-12 & 7.21e-01 & 8.13e-04 & 0:04:42 & 1:00:00 & 1:00:00 \\
$q=300$& 1.88e+4 & 10 & 7 & 4098 & 10236 & 2.54e-14 & 4.13e-01 & 1.00e-11 & 0:08:01 & 1:00:00 & 0:48:14 \\
$p=400$& 1.51e+4* & 10 & 7 & 2854 & 12481 & 3.26e-14 & 2.86e+00 & 1.87e-10 & 0:13:07 & 1:00:00 & 1:00:00 \\
$r=10$& 2.45e+4 & 10 & 7 & 3973 & 7698 & 1.89e-14 & 1.50e+00 & 9.98e-11 & 0:06:59 & 1:00:00 & 0:31:38 \\
& 3.20e+4 & 9 & 7 & 4044 & 5536 & 1.43e-14 & 9.93e-01 & 9.99e-11 & 0:07:42 & 1:00:00 & 0:21:39 \\
\midrule
& 1.00e+1 & 11 & 13 & 812 & 8724 & 1.18e-12 & 9.68e-11 & 3.10e-01 & 0:05:35 & 0:13:51 & 1:00:00 \\
$n=1000$& 1.00e+2 & 11 & 9 & 4404 & 12216 & 8.81e-11 & 1.42e-07 & 7.98e-02 & 0:03:24 & 1:00:00 & 1:00:00 \\
$m=400$& 1.00e+3 & 11 & 7 & 3411 & 12134 & 2.09e-12 & 1.51e+00 & 9.75e-03 & 0:04:44 & 1:00:00 & 1:00:00 \\
$q=400$& 1.00e+4* & 10 & 7 & 3240 & 11594 & 6.29e-14 & 5.67e+00 & 2.02e-05 & 0:10:18 & 1:00:00 & 1:00:00 \\
$p=400$& 1.60e+4 & 9 & 7 & 2276 & 12704 & 4.03e-14 & 3.30e+00 & 3.68e-07 & 0:13:32 & 1:00:00 & 1:00:00 \\
$r=10$& 2.40e+4 & 9 & 7 & 3246 & 12724 & 2.64e-14 & 1.96e+00 & 7.82e-09 & 0:10:22 & 1:00:00 & 1:00:00 \\
& 3.00e+4 & 9 & 7 & 3189 & 12367 & 2.08e-14 & 1.43e+00 & 6.32e-10 & 0:10:38 & 1:00:00 & 1:00:00 \\
\hline
\end{tabular}
\end{adjustbox}
}
\label{table:new_format_update}
\end{table}

It is clear from Figures~\ref{fig:efficiency_synthetic_noVec} and \ref{fig:efficiency_synthetic_vec} that the convergence speed of the PPDNA significantly surpass that of both Nesterov algorithm and ADMM. Figure~\ref{fig:efficiency_synthetic_noVec} reveals that Nesterov algorithm initially descends rapidly. This rapid descent, however, gradually tapers off, which still remains faster than ADMM. Tables~\ref{table:new_format_update} and \ref{table:synthetic_efficiency_vec} confirm that PPDNA consistently outperforms ADMM and Nesterov algorithm, achieving rapid convergence to the desired accuracy within several minutes for both regularization settings. ADMM and Nesterov algorithm, in contrast, consistently take much longer time to converge, and sometimes failed to converge to the desired accuracy within an hour. Note that Nesterov algorithm is more efficient when the penalty parameters are larger, indicating {its} dependence on the regularization {parameters} for {faster} convergence. Overall, the PPDNA demonstrates superior efficiency and robustness in various settings.

\begin{table}[!ht]
\centering
\caption{Time comparison of PPDNA and ADMM with penalty parameters in \eqref{eq:set2} on synthetic data.
Here $\hat{r}$ represents the rank of the estimated matrix coefficient $B^*$ by PPDNA, and ns($\gamma^*$) represents the non-sparsity level of the estimated vector coefficient $\gamma^*$ by PPDNA. The penalty parameters marked with asterisks (*) are used for plotting Figure~\ref{fig:efficiency_synthetic_vec}, as they have been found to recover the true rank of $B$ and the true sparsity level of $\gamma$.
}
{
\setlength{\tabcolsep}{3pt}
\begin{adjustbox}{width=0.99\linewidth}
\begin{tabular}{ccccccccccc}
\hline
& & & & \multicolumn{2}{c}{Iterations} & \multicolumn{2}{c}{$R_{\rm obj}$} & \multicolumn{2}{c}{Time} \\
\cmidrule(r){5-6}\cmidrule(r){7-8}\cmidrule(r){9-10}
Data &$(\rho,\lambda)$ & $\hat{r}$ & ns($\gamma^*$) & P&  A & P & A & P & A \\
\midrule
& (3.46e+1, 2.66e+0) & 11 & 1.00e-2 & 8 & 296 & 5.53e-11 & 9.57e-11 & 0:00:12 & 0:00:52 \\
$n=500$& (3.46e+2, 2.66e+1) & 11 & 1.00e-2 & 7 & 1042 & 5.75e-12 & 9.76e-11 & 0:00:21 & 0:03:50 \\
$m=300$& (3.46e+3, 2.66e+2) & 11 & 1.00e-2 & 7 & 2724 & 4.10e-14 & 9.96e-11 & 0:00:43 & 0:08:45 \\
$q=200$& (2.07e+4, 1.51e+3) & 7 & 1.33e-2 & 6 & 3314 & 5.63e-11 & 9.76e-11 & 0:00:26 & 0:11:58 \\
$p=300$& (2.42e+4, 1.51e+3) & 6 & 5.00e-2 & 6 & 3341 & 4.06e-11 & 9.49e-11 & 0:00:35 & 0:12:04 \\
$r=5$& (2.77e+4, 2.41e+3) & 6 & 3.33e-3 & 6 & 3365 & 2.84e-11 & 9.92e-11 & 0:00:35 & 0:19:30 \\
& (3.11e+4, 2.41e+3) & 6 & 6.67e-3 & 6 & 3404 & 1.97e-11 & 9.19e-11 & 0:00:35 & 0:19:21 \\
& (3.46e+4, 2.66e+3)* & 5 & 1.00e-2 & 6 & 3424 & 1.31e-11 & 9.69e-11 & 0:00:27 & 0:13:04 \\
\midrule
& (5.17e+1, 3.69e+0) & 11 & 2.00e-2 & 8 & 366 & 3.36e-11 & 9.50e-11 & 0:00:20 & 0:01:27 \\
$n=500$& (5.17e+2, 3.69e+1) & 11 & 2.00e-2 & 7 & 1234 & 3.13e-12 & 9.97e-11 & 0:00:31 & 0:06:29 \\
$m=300$& (5.17e+3, 3.69e+2) & 11 & 2.00e-2 & 7 & 2903 & 4.17e-14 & 9.74e-11 & 0:01:02 & 0:14:16 \\
$q=300$& (3.82e+4, 3.10e+3) & 7 & 3.33e-3 & 6 & 3506 & 3.61e-11 & 9.08e-11 & 0:00:44 & 0:19:59 \\
$p=300$& (4.13e+4, 3.40e+3) & 7 & 3.33e-3 & 6 & 3547 & 2.95e-11 & 9.57e-11 & 0:00:45 & 0:20:23 \\
$r=5$& (4.44e+4, 3.71e+3) & 6 & 3.33e-3 & 6 & 3578 & 2.38e-11 & 9.53e-11 & 0:00:52 & 0:25:27 \\
& (4.65e+4, 3.71e+3) & 6 & 3.33e-3 & 6 & 3596 & 2.06e-11 & 9.82e-11 & 0:00:50 & 0:20:41 \\
& (5.17e+4, 3.69e+3)* & 5 & 1.00e-2 & 6 & 3647 & 1.40e-11 & 9.58e-11 & 0:00:38 & 0:20:28 \\
\midrule
& (6.59e+1, 4.98e+0) & 15 & 1.50e-2 & 8 & 349 & 1.39e-11 & 8.91e-11 & 0:00:58 & 0:03:42 \\
$n=1000$& (6.59e+2, 4.98e+1) & 15 & 1.50e-2 & 7 & 1265 & 2.10e-12 & 9.73e-11 & 0:01:13 & 0:13:59 \\
$m=400$& (6.59e+3, 4.98e+2) & 15 & 1.50e-2 & 7 & 2925 & 5.88e-14 & 9.73e-11 & 0:04:51 & 0:41:28 \\
$q=300$& (4.70e+4, 3.32e+3) & 12 & 2.00e-2 & 6 & 3497 & 7.51e-11 & 9.90e-11 & 0:04:22 & 0:54:47 \\
$p=400$& (5.65e+4, 3.32e+3) & 11 & 3.00e-2 & 6 & 2538 & 5.22e-11 & 5.35e-02 & 0:06:00 & 1:00:00 \\
$r=10$& (6.59e+4, 3.87e+3) & 10 & 3.25e-2 & 6 & 3605 & 3.65e-11 & 8.98e-11 & 0:04:12 & 0:56:43 \\
& (6.59e+4, 4.98e+3)* & 10 & 1.00e-2 & 6 & 3607 & 3.63e-11 & 9.43e-11 & 0:04:23 & 0:58:24 \\
& (7.53e+4, 4.42e+3) & 9 & 3.25e-2 & 6 & 3682 & 2.52e-11 & 9.79e-11 & 0:04:05 & 0:59:57 \\
\midrule
& (9.00e+1, 6.27e+0) & 17 & 1.25e-2 & 8 & 425 & 8.40e-12 & 9.79e-11 & 0:01:52 & 0:09:43 \\
$n=1000$& (9.00e+2, 6.27e+1) & 17 & 1.25e-2 & 7 & 1483 & 1.26e-12 & 9.90e-11 & 0:02:26 & 0:27:24 \\
$m=400$& (9.00e+3, 6.27e+2) & 15 & 1.25e-2 & 7 & 3066 & 5.49e-14 & 9.82e-11 & 0:06:05 & 0:58:29 \\
$q=400$& (7.00e+4, 4.67e+3) & 12 & 1.50e-2 & 6 & 2845 & 5.38e-11 & 1.87e-02 & 0:05:33 & 1:00:00 \\
$p=400$& (9.00e+4, 6.01e+3) & 10 & 1.50e-2 & 6 & 2019 & 2.90e-11 & 2.76e-01 & 0:06:58 & 1:00:00 \\
$r=10$& (9.00e+4, 6.27e+3)* & 10 & 1.00e-2 & 6 & 2517 & 2.90e-11 & 2.75e-01 & 0:04:51 & 1:00:00 \\
& (1.00e+5, 5.34e+3) & 9 & 4.75e-2 & 6 & 2425 & 2.12e-11 & 2.01e-01 & 0:07:53 & 1:00:00 \\
& (1.20e+5, 5.34e+3) & 7 & 9.25e-2 & 6 & 2845 & 1.08e-11 & 8.81e-02 & 0:05:04 & 1:00:00 \\
\hline
\end{tabular}
\end{adjustbox}
}
\label{table:synthetic_efficiency_vec}
\end{table}

\begin{figure}[!ht]
    \centering
\includegraphics[width=0.24\linewidth]{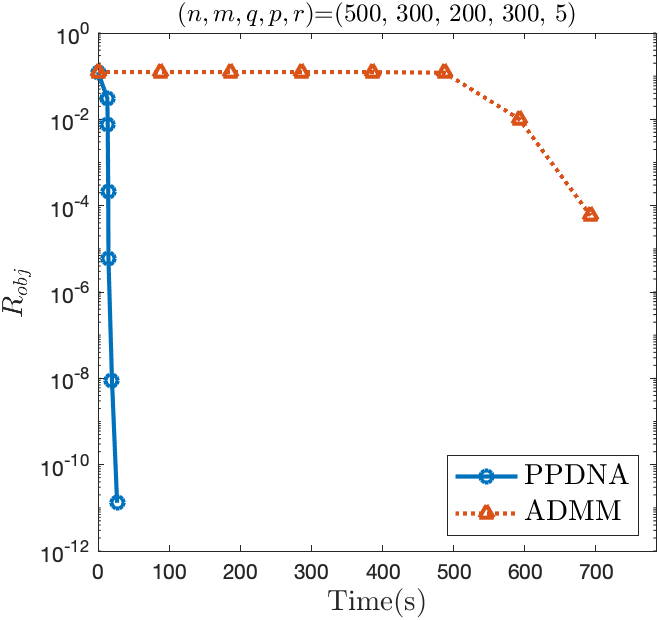}
        \includegraphics[width=0.24\linewidth]{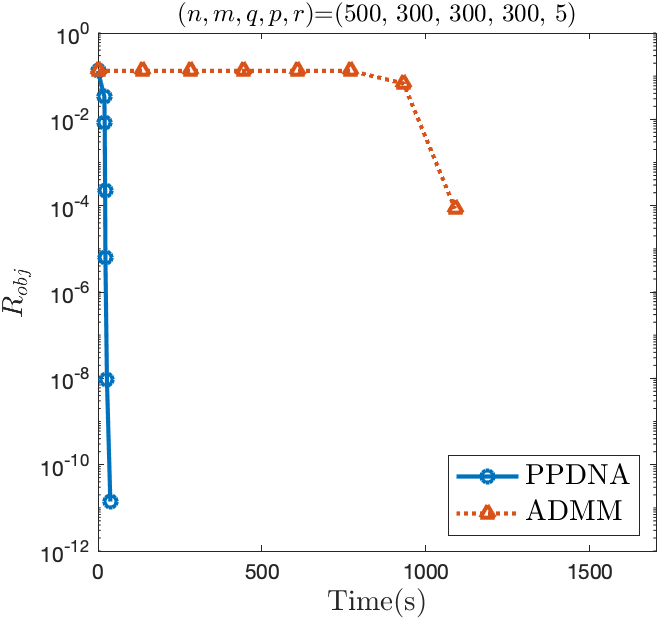}
        \includegraphics[width=0.24\linewidth]{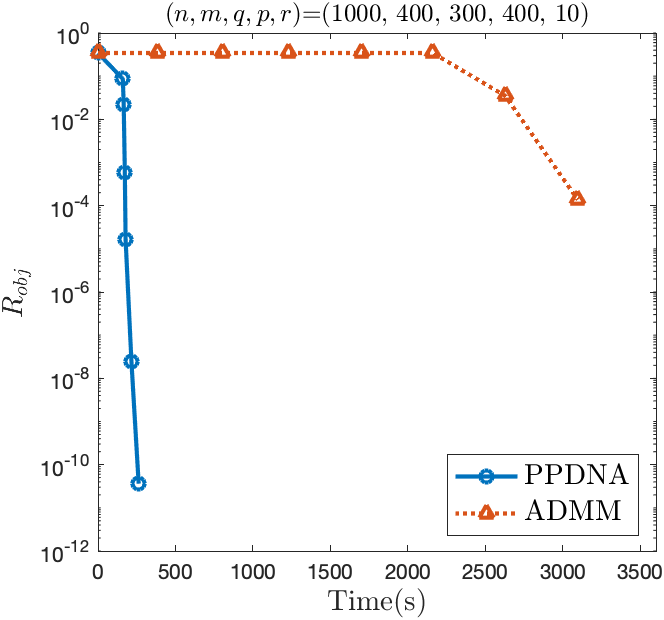}
        \includegraphics[width=0.24\linewidth]{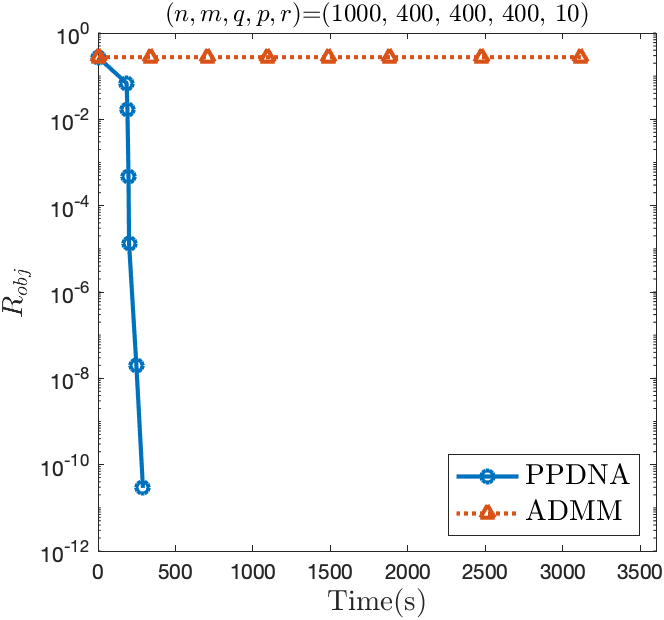}
    \caption{Time vs. $R_{\rm obj}$ (see \eqref{eq:Robj})  of PPDNA and ADMM on synthetic data. The penalty parameters are taken from Table~\ref{table:synthetic_efficiency_vec} with {asterisks} (*).
    }
    \label{fig:efficiency_synthetic_vec}
\end{figure}

\subsubsection{COVID-19 Data}
In this subsection, we conduct {numerical experiments} on the COVID-19 data \citep{COVID}, which can be downloaded from \url{https://health.google.com/covid-19/open-data/}. This data offers a comprehensive daily time-series of COVID-19 cases, deaths, recoveries, tests, vaccinations, and hospitalizations across more than 230 countries, 760 regions, and 12,000 lower-level administrative divisions. In our analysis, we concentrate on 181 regions over a period of 911 days from January 1, 2020, to June 30, 2022.

The response variable $y$ represents the total count of confirmed cases within each region during this period. Each matrix-valued data $X_i$, for all $i\in [181]$ encompasses various COVID-19 related information, each of which fluctuates over the 911-day period, including:

\begin{enumerate}[itemsep=1pt, topsep=2pt, partopsep=0pt, parsep=0pt]
\item Government response implemented daily, which contains 20 factors such as school closures, gathering restrictions, stay-at-home mandates, and income support initiatives.
\item Mobility, which contains 6 factors reflecting different aspects of people's movement patterns, such as percentage {changes} in visits to restaurants, cafes, shopping centers, theme parks, museums, libraries, and movie theaters compared to {the} baseline.
\item Weather, which contains 7 weather related factors including recorded hourly average, minimum and maximum temperature, rainfall and snowfall during the entire day.
\item Vaccination trends, which contains 30 factors related to deployment and administration of COVID-19 vaccines, such as {the} count of new persons which have received one or more doses.
\end{enumerate}

Therefore, each $X_i$ is of size $911\times63$, resulting from 911-day framework and a total of 63 measurements derived from the aforementioned information. Additionally, the vector-valued sample {$z_i \in\mathbb{R}^{13}$} encapsulates 13 health-related characteristics for each region that remain relatively constant or exhibit minimal variation during the 911-day period. These characteristics include factors such as life expectancy, hospital bed capacity, prevalence of diabetes and smoking, and the number of nurses. During the data pre-processing, we replaced the missing values by 0. In addition, we standardized both the matrices $X_i$'s and the  matrix $Z$ column-wisely. The response variable $y$ is also standardized. {For clarity, we specify the dimensions as follows: $m=911$, $q=63$, $p=13$, and $n=181$.}

Table~\ref{table:revised_COVID_efficiency} compares PPDNA and ADMM on COVID-19 data across a wide range of $\alpha_1$ and $\alpha_2$ values in \eqref{eq:set2}, resulting in various values of $\rho$ and $\lambda$. Additionally, Table~\ref{table:updated_synthetic_efficiency_comparison} includes Nesterov algorithm alongside PPDNA and ADMM in the comparison, with a broad selection of $\alpha$ values in \eqref{eq:set1}, highlighting the performance without regularization on the vector coefficient.

We can see from Table~\ref{table:revised_COVID_efficiency} that PPDNA consistently outperforms ADMM in terms of computational time. PPDNA typically achieves the desired accuracy in less than {one} minute, whereas ADMM often takes several minutes to converge.
Additionally, in Table~\ref{table:updated_synthetic_efficiency_comparison},  Nesterov algorithm exhibits even slower convergence and occasionally fails to descend,  terminating prematurely.

By combining the findings presented in Section \ref{sec:efficiency}, we can safely conclude that PPDNA emerges as a significantly more efficient and robust algorithm  compared to ADMM and Nesterov algorithm, for estimating both matrix and vector variables in problem \eqref{eq: p}.

\begin{table}[H]
  \centering
  \caption{Time comparison of  PPDNA and ADMM with penalty parameters in \eqref{eq:set2} on COVID-19 data. Here $\hat{r}$ represents the rank of the estimated matrix coefficient $B^*$ by PPDNA,  ns($\gamma^*$) represents the non-sparsity level of the estimated vector coefficient $\gamma^*$ by PPDNA, and $R_{\rm obj}$ is given by  \eqref{eq:Robj}.
    }
    \begin{tabular}{cccccccccc}
    \hline
   & & & &\multicolumn{2}{c}{Iterations} & \multicolumn{2}{c}{$R_{\rm obj}$} & \multicolumn{2}{c}{Time}  \\
    \cmidrule(r){5-6}     \cmidrule(r){7-8}     \cmidrule(r){9-10}
     $\alpha_1$&$\alpha_2$ & $\hat{r}$ & ns($\gamma^*$) & P & A & P & A & P & A \\
     \midrule
     \multirow{3}{*}{1e-1} & 1e-1 & 3 & 7.69e-2 & 10 & 2747 & 1.74e-11 & 9.80e-11 & 0:00:16 & 0:07:51 \\
     & 5e-2 & 3 & 4.62e-1 & 9 & 2715 & 4.89e-11 & 9.73e-11 & 0:00:17 & 0:07:44 \\
     & 1e-2 & 3 & 9.23e-1 & 9 & 2779 & 9.91e-11 & 9.84e-11 & 0:00:18 & 0:07:50 \\
     \midrule
     \multirow{3}{*}{1e-2} & 1e-2 & 4 & 1.54e-1 & 11 & 1016 & 9.69e-11 & 9.73e-11 & 0:00:15 & 0:02:34 \\
     & 5e-3 & 4 & 6.15e-1 & 12 & 962 & 3.11e-12 & 9.91e-11 & 0:00:16 & 0:02:27 \\
     & 1e-3 & 4 & 8.46e-1 & 13 & 1227 & 8.80e-12 & 9.89e-11 & 0:00:19 & 0:03:08 \\
     \midrule
     \multirow{3}{*}{1e-3} & 1e-3 & 5 & 1.54e-1 & 11 & 311 & 2.14e-12 & 9.14e-11 & 0:00:16 & 0:00:47 \\
     & 5e-4 & 4 & 6.92e-1 & 13 & 544 & 3.82e-12 & 9.93e-11 & 0:00:24 & 0:01:29 \\
     & 1e-4 & 4 & 8.46e-1 & 17 & 1615 & 2.36e-11 & 9.92e-11 & 0:00:19 & 0:03:01 \\
     \midrule
     \multirow{3}{*}{1e-4} & 1e-4 & 5 & 1.54e-1 & 19 & 447 & 8.51e-12 & 9.87e-11 & 0:00:11 & 0:00:45 \\
     & 5e-5 & 4 & 6.92e-1 & 36 & 2111 & 8.55e-11 & 9.98e-11 & 0:00:17 & 0:03:41 \\
     & 1e-5 & 4 & 8.46e-1 & 18 & 4830 & 5.51e-11 & 9.98e-11 & 0:00:22 & 0:07:40 \\
     \midrule
     \multirow{3}{*}{1e-5} & 1e-5 & 5 & 1.54e-1 & 42 & 2441 & 9.57e-11 & 9.96e-11 & 0:00:16 & 0:03:35 \\
     & 5e-6 & 4 & 6.92e-1 & 100 & 6460 & 0.00e+00 & 9.99e-11 & 0:00:31 & 0:09:20 \\
     & 1e-6 & 4 & 8.46e-1 & 100 & 9903 & 0.00e+00 & 9.54e-11 & 0:00:38 & 0:14:22 \\
    \hline
    \end{tabular}
  \label{table:revised_COVID_efficiency}
\end{table}

\begin{table}[H]
\centering
\caption{Time comparison of PPDNA, ADMM, and Nesterov algorithm with penalty parameters in \eqref{eq:set1} on COVID-19 data.
Here $\hat{r}$ represents the rank of the estimated matrix coefficient $B^*$ by PPDNA, and $R_{\rm obj}$ is given by  \eqref{eq:Robj}.
}
{
\setlength{\tabcolsep}{3pt}
\begin{tabular}{ccccccccccc}
\hline
& & \multicolumn{3}{c}{Iterations} & \multicolumn{3}{c}{$R_{\rm obj}$} & \multicolumn{3}{c}{Time}  \\
\cmidrule(r){3-5}\cmidrule(r){6-8}\cmidrule(r){9-11}
$\alpha$ & $\hat{r}$ & P & A & N & P & A & N & P & A & N \\
\midrule
1e-1 & 3 &  10 & 3188 & 20080 & 1.30e-11 & 9.90e-11 & 1.00e-10 & 0:00:18 & 0:09:32 & 0:05:14\\
1e-2 & 4 &  12 & 1783 & 228555 & 2.19e-11 & 9.95e-11 & 6.97e-07 & 0:00:24 & 0:04:33 & 1:00:00\\
1e-3 & 4 &  15 & 1750 & 218431 & 5.95e-11 & 9.95e-11 & 5.27e-03 & 0:00:28 & 0:03:21 & 1:00:00\\
1e-4 & 4 &  27 & 5981 & 229822 & 9.46e-11 & 9.99e-11 & 1.06e-02 & 0:00:26 & 0:09:21 & 1:00:00\\
1e-5 & 4 & 97 & 14030 & 2991 & 6.07e-11 & 9.79e-11 & 4.30e-03 & 0:00:45 & 0:19:17 & 0:00:51\\
\hline
\end{tabular}
}
\label{table:updated_synthetic_efficiency_comparison}
\end{table}

\section{Conclusion and Discussion}\label{sec:conclusion}
In this paper, we proposed a class of regularized regression models that can incorporate the matrix and vector predictors simultaneously with  general convex penalty functions. The $n$-consistency and $\sqrt{n}$-consistency of the estimator obtained from the optimization model were established when nuclear norm and lasso norm penalties are employed.  We proposed a highly efficient dual semismooth Newton method based preconditioned proximal point algorithm for solving the general model, which fully exploits the second-order information. Extensive numerical experiments were conducted to demonstrate the practical and general applicability of our proposed algorithm, as well as its efficiency {and robustness} compared to the state-of-the-art algorithms for solving regularized regression models. {While this paper primarily focuses on linear regression problems, our algorithm can  handle other problems, including logistic regression and Poisson regression problems.
Expanding our consistency analysis for other loss and penalty functions remains an open question, as we currently address a special case in Section~\ref{sec:consistency}.
Additionally, when dealing with possibly non-convex loss and penalty functions, we may need to  combine our algorithmic framework with difference-of-convex (d.c.) tools and majorization techniques to design d.c. type algorithms.
}


\appendix
\section{Proof of Theorem~\ref{thm: n-consistency}}\label{appendix: sec1}
\textbf{Proof of Theorem~\ref{thm: n-consistency}}
For any given $(U,\beta) \in \mathbb{R}^{m\times q}\times \mathbb{R}^p$, we have that
\begin{align*}
&Z_n(U,\beta) - Z(U,\beta)\\
=& \frac{1}{n}\sum_{i=1}^n (y_i\!-\!\langle X_i,U\rangle \!-\! \langle z_i,\beta\rangle)^2 +  \frac{\rho_n}{n}\|U\|_* + \frac{\lambda_n}{n}\|\beta\|_1 -{\rm vec}(U\!- \!B)^{\intercal} C {\rm vec}(U\!-\! B) \\
& - 2 {\rm vec}(U- B)^{\intercal}  H (\beta-\gamma)
- (\beta-\gamma)^{\intercal} D (\beta-\gamma)
- \rho_0\|U\|_* - \lambda_0\|\beta\|_1\\
=& \frac{1}{n}\sum_{i=1}^n (\varepsilon_i-\langle X_i,U-B\rangle - \langle z_i,\beta-\gamma\rangle)^2
- {\rm vec}(U- B)^{\intercal} C {\rm vec}(U- B) \\
& \!-\! 2 {\rm vec}(U\!-\! B)^{\intercal}  H (\beta\!-\!\gamma) \!-\! (\beta\!-\!\gamma)^{\intercal} D (\beta\!-\!\gamma)\!+\! \left(\frac{\rho_n}{n}\! -\!\rho_0\right)\|U\|_* \!+\! \left(\frac{\lambda_n}{n}\!- \!\lambda_0 \right)\|\beta\|_1\\
=& \frac{1}{n}\sum_{i=1}^n \varepsilon_i^2 + \frac{1}{n}\sum_{i=1}^n \left( \langle X_i,U-B\rangle\right)^2 + \frac{1}{n}\sum_{i=1}^n \left( \langle z_i,\beta-\gamma\rangle \right)^2 - \frac{2}{n}\sum_{i=1}^n  \varepsilon_i \langle X_i,U-B\rangle \\
&- \frac{2}{n}\sum_{i=1}^n  \varepsilon_i \langle z_i,\beta-\gamma\rangle + \frac{2}{n}\sum_{i=1}^n  \langle X_i,U\!-\!B\rangle\langle z_i,\beta - \gamma\rangle-{\rm vec}(U\!-\! B)^{\intercal} C {\rm vec}(U\!-\! B) \\
&  \!-\! 2 {\rm vec}(U\!-\! B)^{\intercal}  H (\beta\!-\!\gamma)\! -\! (\beta\!-\!\gamma)^{\intercal} D (\beta\!-\!\gamma)\!+\! \left(\frac{\rho_n}{n} \!-\!\rho_0\right)\|U\|_* \!+\! \left(\frac{\lambda_n}{n}\!-\! \lambda_0 \right)\|\beta\|_1 \\
 =& \frac{1}{n}\sum_{i=1}^n \varepsilon_i^2 + {\rm vec}(U- B)^{\intercal} (C_n - C) {\rm vec}(U- B) + (\beta-\gamma)^{\intercal} (D_n-D) (\beta-\gamma) \\
 &+ 2  {\rm vec}(U- B)^{\intercal}  (H_n-H) (\beta-\gamma) -\frac{2}{n}\sum_{i=1}^n  \varepsilon_i \left( \langle X_i,U-B\rangle +  \langle z_i,\beta-\gamma\rangle \right) \\
&  + \left(\frac{\rho_n}{n} -\rho_0\right)\|U\|_* + \left(\frac{\lambda_n}{n}- \lambda_0 \right)\|\beta\|_1.
\end{align*}
By the weak law of large numbers, we have that $\frac{1}{n} \sum_{i=1}^n \varepsilon_i^2 \to \mathbb{E}[\varepsilon_i^2] = \sigma^2 $ in probability as $n\rightarrow \infty$.

Moreover, it can be seen that
\begin{align*}
\mathbb{E}\!\left[ \frac{1}{n}\!\sum_{i=1}^n \! \varepsilon_i \left( \langle X_i,U\!-\!B\rangle +  \langle z_i,\beta\!-\!\gamma\rangle \right) \right] \!=\! \frac{1}{n}\sum_{i=1}^n \mathbb{E}[\varepsilon_i]\left( \langle X_i,U\!-\!B\rangle +  \langle z_i,\beta\!-\!\gamma\rangle \right) \!=\!0,
\end{align*}
\begin{align*}
& {\rm Var}\left[ \frac{1}{n}\sum_{i=1}^n  \varepsilon_i \left( \langle X_i,U-B\rangle +  \langle z_i,\beta-\gamma\rangle \right) \right] = \mathbb{E}\left[ \frac{1}{n^2}\sum_{i=1}^n \varepsilon_i^2 \left( \langle X_i,U-B\rangle+ \langle z_i,\beta-\gamma\rangle\right)^2 \right]\\
=&\frac{1}{n^2}\sum_{i=1}^n \mathbb{E}[\varepsilon_i^2] \left( \langle X_i,U-B\rangle+ \langle z_i,\beta-\gamma\rangle\right)^2= \frac{\sigma^2}{n}\begin{pmatrix}
{\rm vec}(U- B)\\
\beta-\gamma
\end{pmatrix}^{\intercal}
S_n
\begin{pmatrix}
{\rm vec}(U- B)\\
\beta-\gamma
\end{pmatrix}
:= \frac{\sigma^2}{n} \hat{\sigma}_n^2.
\end{align*}
Here, $\hat{\sigma}_n^2 > 0$ for large $n$ because $\hat{\sigma}_n^2$ converges to a positive number, as long as $(U,\beta) \neq (B,\gamma)$. In cases where they are equal, $\frac{1}{n}\sum_{i=1}^n  \varepsilon_i \left( \langle X_i,U\!\!-\!\!B\rangle +  \langle z_i,\beta\!\!-\!\!\gamma\rangle \right)$ is zero.
By Chebyshev's inequality, for any given $\epsilon > 0$, we have that
\begin{align*}
\mathbb{P}\left( \Big| \frac{1}{n}\sum_{i=1}^n  \varepsilon_i \left( \langle X_i,U-B\rangle +  \langle z_i,\beta-\gamma\rangle \right)  \Big| \geq \epsilon\right) \leq \frac{\sigma^2 \hat{\sigma}_n^2}{\epsilon^2 n},
\end{align*}
which further implies that
$
\lim_{n\to  \infty} \mathbb{P}\left( \Big| \frac{1}{n}\sum_{i=1}^n  \varepsilon_i \left( \langle X_i,U\!\!-\!\!B\rangle \!+\!  \langle z_i,\beta\!-\!\gamma\rangle \right) \Big| \!\!\geq\! \!\epsilon\right) \!=\!0.
$
Therefore, we have proved that
$
Z_n(U,\beta) - Z(U,\beta)\rightarrow \sigma^2
$
in probability as $n\rightarrow \infty$. According to \citet{pollard1991asymptotics}, we further have
\begin{align}
\sup_{(U,\beta)\in K} \ \Big| Z_n(U,\beta) - Z(U,\beta) - \sigma^2 \Big| \rightarrow 0 \label{eq: Z_converge_on_k}
\end{align}
in probability as $n\rightarrow \infty$, for any compact set $K\subset \mathbb{R}^{m\times q}\times \mathbb{R}^p$.

Next, we denote
$
Z_n^{(0)}(U,\beta):=\frac{1}{n}\sum_{i=1}^n (y_i-\langle X_i,U\rangle - \langle z_i,\beta\rangle)^2,
$
and assume that it is minimized at $(\widehat{B}_n^{(0)},\hat{\gamma}_n^{(0)})$.

Then we have
\begin{align*}
&\frac{1}{n}\sum_{i=1}^n \left( y_i-\langle X_i,\widehat{B}_n^{(0)}\rangle - \langle z_i,\hat{\gamma}_n^{(0)}\rangle \right)^2 \leq \frac{1}{n}\sum_{i=1}^n (y_i-\langle X_i,B\rangle - \langle z_i,\gamma\rangle)^2,\\
&\frac{1}{n}\sum_{i=1}^n \left(\varepsilon_i-\langle X_i,\widehat{B}_n^{(0)}-B\rangle - \langle z_i,\hat{\gamma}_n^{(0)}-\gamma\rangle \right)^2 \leq \frac{1}{n}\sum_{i=1}^n \varepsilon_i^2.
\end{align*}
Then we can see that
\begin{align*}
& \begin{pmatrix}
{\rm vec} (\widehat{B}_n^{(0)} - B) \\ \hat{\gamma}_n^{(0)}-\gamma
\end{pmatrix}^{\intercal}
S_n
\begin{pmatrix}
{\rm vec} (\widehat{B}_n^{(0)}- B) \\ \hat{\gamma}_n^{(0)}-\gamma
\end{pmatrix} \\
=& {\rm vec} (\widehat{B}_n^{(0)} \!\!- \!\!B)^{\intercal} \! C_n {\rm vec} (\widehat{B}_n^{(0)} \!\!-\!\! B) \!+\! (\hat{\gamma}_n^{(0)}\!\!-\!\!\gamma)^{\intercal} \!D_n (\hat{\gamma}_n^{(0)}\!\!-\!\!\gamma) \!+\! 2{\rm vec} (\widehat{B}_n^{(0)} \!\!- \!\!B)^{\intercal} \! H_n(\hat{\gamma}_n^{(0)}\!\!-\!\!\gamma)\\
=&\frac{1}{n}\!\!\sum_{i=1}^n \!\!\left(\langle X_i,\widehat{B}_n^{(0)}\!\!-\!\!B\rangle \!+\! \langle z_i,\hat{\gamma}_n^{(0)}\!\!-\!\!\gamma\rangle \right)^2
\!\!=\!\! \frac{1}{n}\!\!\sum_{i=1}^n \!\!\left(\langle X_i,\widehat{B}_n^{(0)}\!\!-\!\!B\rangle\! +\! \langle z_i,\hat{\gamma}_n^{(0)}\!\!-\!\!\gamma\rangle \!-\! \varepsilon_i\!+\! \varepsilon_i \right)^2\\
\leq &\frac{2}{n}\sum_{i=1}^n \left(\langle X_i,\widehat{B}_n^{(0)}-B\rangle + \langle z_i,\hat{\gamma}_n^{(0)}-\gamma\rangle - \varepsilon_i \right)^2 + \frac{2}{n}\sum_{i=1}^n\varepsilon_i^2\leq \frac{4}{n}\sum_{i=1}^n\varepsilon_i^2.
\end{align*}
By Assumption~\ref{assu: regularity} and the positive definiteness of $S$, there exists $\eta>0$ such that $S_n\succeq  \eta I_{mq+p}$ for  large $n$. Therefore,
\begin{align}
\eta \left( \| \widehat{B}_n^{(0)} -  B\|_F^2+ \|\hat{\gamma}_n^{(0)}-\gamma\|^2 \right)\leq \frac{4}{n}\sum_{i=1}^n\varepsilon_i^2.\label{eq: Bgamma_bound}
\end{align}
Since $\sum_{i=1}^n (\varepsilon_i/\sigma)^2$ follows the chi-squared distribution with $n$ degrees of freedom, we have
$ \mathbb{E}\left[ \frac{1}{n} \sum_{i=1}^n \varepsilon_i^2 \right]= \sigma^2$ and $ {\rm Var} \left[ \frac{1}{n} \sum_{i=1}^n \varepsilon_i^2 \right]= \frac{2\sigma^4}{n}$.
Then from \eqref{eq: Bgamma_bound} and Chebyshev’s inequality,
\begin{equation*}
\begin{aligned}
& \mathbb{P}\left( \| \widehat{B}_n^{(0)} -  B\|_F^2+ \|\hat{\gamma}_n^{(0)}-\gamma\|^2 > \frac{4}{\eta} \left(\sigma^2+1 \right)\right)\leq
\mathbb{P}\left( \frac{1}{n} \sum_{i=1}^n \varepsilon_i^2 > \sigma^2+1 \right) \\
&\qquad \leq  \mathbb{P}\left( \Big| \frac{1}{n} \sum_{i=1}^n \varepsilon_i^2-\sigma^2\Big| \geq 1\right) \leq \frac{2\sigma^4}{n},
\end{aligned}
\end{equation*}
which further implies that
\begin{align}
\lim_{n\to\infty}\mathbb{P}\left( \| \widehat{B}_n^{(0)} -  B\|_F^2+ \|\hat{\gamma}_n^{(0)}-\gamma\|^2 > \frac{4}{\eta} \left(\sigma^2+1 \right)\right)=0. \label{eq: limb0gamma0}
\end{align}
This {means} that $\{(\widehat{B}_n^{(0)},\hat{\gamma}_n^{(0)})\}$ is bounded in probability, i.e., $(\widehat{B}_n^{(0)},\hat{\gamma}_n^{(0)}) = O_P(1)$.
Furthermore, due to the fact that $Z_n(\widehat{B}_n,\hat{\gamma}_n)\leq Z_n(\widehat{B}_n^{(0)},\hat{\gamma}_n^{(0)}) $ and $Z_n^{(0)}(\widehat{B}_n^{(0)},\hat{\gamma}_n^{(0)})\leq Z_n^{(0)}(\widehat{B}_n,\hat{\gamma}_n)$, we have
\begin{align*}
&\quad \frac{1}{n}\sum_{i=1}^n \left(y_i-\langle X_i,\widehat{B}_n\rangle - \langle z_i,\hat{\gamma}_n\rangle \right)^2 +  \frac{\rho_n}{n}\|\widehat{B}_n\|_* + \frac{\lambda_n}{n}\|\hat{\gamma}_n\|_1 \\
&\leq \frac{1}{n}\sum_{i=1}^n \left(y_i-\langle X_i,\widehat{B}_n^{(0)}\rangle - \langle z_i,\hat{\gamma}_n^{(0)}\rangle \right)^2 +  \frac{\rho_n}{n}\|\widehat{B}_n^{(0)}\|_* + \frac{\lambda_n}{n}\|\hat{\gamma}_n^{(0)}\|_1\\
&\leq \frac{1}{n}\sum_{i=1}^n \left(y_i-\langle X_i,\widehat{B}_n\rangle - \langle z_i,\hat{\gamma}_n\rangle \right)^2 +  \frac{\rho_n}{n}\|\widehat{B}_n^{(0)}\|_* + \frac{\lambda_n}{n}\|\hat{\gamma}_n^{(0)}\|_1,
\end{align*}
which means
\begin{align}
\frac{\rho_n}{n}\|\widehat{B}_n\|_* + \frac{\lambda_n}{n}\|\hat{\gamma}_n\|_1 \leq \frac{\rho_n}{n}\|\widehat{B}_n^{(0)}\|_* + \frac{\lambda_n}{n}\|\hat{\gamma}_n^{(0)}\|_1. \label{eq: bgamma_b0gamma0}
\end{align}
Therefore, the sequence $\{(\widehat{B}_n,\hat{\gamma}_n)\}$ is also bounded in probability: $ (\widehat{B}_n,\hat{\gamma}_n) = O_P(1)$.

Lastly, for any compact set $K\subset \mathbb{R}^{m\times q}\times \mathbb{R}^p$, we have
\begin{align*}
\sup_{(U,\beta)\in K} \ \Big| Z_n(U,\beta) - Z(U,\beta) - \sigma^2 \Big| \geq \Big| \inf_{(U,\beta)\in K} Z_n(U,\beta) - \inf_{(U,\beta)\in K} Z(U,\beta) -\sigma^2 \Big|,
\end{align*}
due to the following two inequalities:
\begin{align*}
&\inf_{(U,\beta)\in K} \!\! \! Z_n(U,\beta) \!- \!\!\!\! \inf_{(U,\beta)\in K} \!\! Z(U,\beta) \!-\!\sigma^2 \!=\! -\!\!\!\sup_{(U,\beta)\in K} \! (-Z_n(U,\beta)) \!+\!\!\!\! \sup_{(U,\beta)\in K} \!\!\!(-Z(U,\beta)\!-\!\sigma^2)\\
&\qquad \leq \sup_{(U,\beta)\in K} (Z_n(U,\beta) -Z(U,\beta)-\sigma^2)\leq \sup_{(U,\beta)\in K} \Big| Z_n(U,\beta) -Z(U,\beta)-\sigma^2\Big|, \\
&-\!\!\!\!\inf_{(U,\beta)\in K} \!\!\!Z_n(U,\beta) \!+ \!\!\!\!\inf_{(U,\beta)\in K} Z(U,\beta) \!+\!\sigma^2 \!=\!\!\! \sup_{(U,\beta)\in K} \!\!\!(-Z_n(U,\beta)) \!-\!\!\! \sup_{(U,\beta)\in K} \!\!\!(-Z(U,\beta)\!-\!\sigma^2)\\
&\qquad \leq \sup_{(U,\beta)\in K} (-Z_n(U,\beta) +Z(U,\beta)+\sigma^2)\leq \sup_{(U,\beta)\in K} \Big| Z_n(U,\beta) -Z(U,\beta)-\sigma^2\Big|.
\end{align*}
By \eqref{eq: limb0gamma0} and \eqref{eq: bgamma_b0gamma0}, there exists $M>0$  such that
\begin{equation}\label{eq: bound2}
\lim_{n\to \infty} \mathbb{P}\left( \|\widehat{B}_n\|^2 + \|\hat{\gamma}_n\|^2>M\right) =0.
\end{equation}
Denote $K_n:=\{(\widehat{B}_n,\hat{\gamma}_n),(\widehat{B},\hat{\gamma})\}$ and $K_0:=\{(B,\gamma) \in \mathbb{R}^{m\times q}\times \mathbb{R}^p\mid  \|B\|^2 + \|\gamma\|^2 \leq M\} \cup \{(\widehat{B},\hat{\gamma})\}$. Then we have that
\begin{align*}
&|Z(\widehat{B}_n,\! \hat{\gamma}_n)\! - \!Z(\widehat{B},\!\hat{\gamma})|
\!\leq \! |Z_n(\widehat{B}_n,\! \hat{\gamma}_n) \!- \!Z(\widehat{B}_n,\! \hat{\gamma}_n)\! -\! \sigma^2 |\! +\! |Z_n(\widehat{B}_n,\hat{\gamma}_n) \!- \!Z(\widehat{B},\hat{\gamma}) \!-\! \sigma^2 |\\
&\quad \leq \sup_{(U,\beta)\in K_n} \!\!|Z_n(U,\beta)\!- \!Z(U,\beta) \!-\! \sigma^2|   \!+\!|\inf_{(U,\beta) \in K_n} \!\!Z_n(U,\beta) \!-\! \inf_{(U,\beta) \in K_n} \!\!Z(U,\beta) \!-\! \sigma^2|\\
&\quad \leq  2\sup_{(U,\beta)\in K_n} |Z_n(U,\beta) - Z(U,\beta) - \sigma^2|.
\end{align*}
Therefore, for any $\epsilon > 0$, we can see that
\begin{align*}
&\mathbb{P}\left( |Z(\widehat{B}_n,\hat{\gamma}_n) - Z(\widehat{B},\hat{\gamma})|\geq \epsilon\right)\\
=&\mathbb{P}\left( |Z(\widehat{B}_n,\hat{\gamma}_n) - Z(\widehat{B},\hat{\gamma})|\geq \epsilon, \|\widehat{B}_n\|^2 + \|\hat{\gamma}_n\|^2\leq M \right)\\
&+ \mathbb{P}\left( |Z(\widehat{B}_n,\hat{\gamma}_n) - Z(\widehat{B},\hat{\gamma})|\geq \epsilon, \|\widehat{B}_n\|^2 + \|\hat{\gamma}_n\|^2> M \right)\\
\leq&  \mathbb{P}\left( \|\widehat{B}_n\|^2 + \|\hat{\gamma}_n\|^2\leq  M \right) \mathbb{P}\left( |Z(\widehat{B}_n,\hat{\gamma}_n) - Z(\widehat{B},\hat{\gamma})|\geq \epsilon \middle\vert \|\widehat{B}_n\|^2 + \|\hat{\gamma}_n\|^2\leq M \right)\\
& + \mathbb{P}\left( \|\widehat{B}_n\|^2 + \|\hat{\gamma}_n\|^2> M \right)\\
\leq& \mathbb{P}\left( \!\!\|\widehat{B}_n\|^2 \!\!+\!\! \|\hat{\gamma}_n\|^2\!\!\leq\!  \!M \!\!\right) \mathbb{P}\left( \!\!\sup_{(U,\beta)\in K_n} \!\!\!\!|Z_n(U,\beta) \!-\! Z(U,\beta) \!-\! \sigma^2|\!\geq\! \frac{\epsilon}{2} \middle\vert \|\widehat{B}_n\|^2 \!\!+ \!\!\|\hat{\gamma}_n\|^2\!\leq \! M  \!\!\right) \\
&+ \mathbb{P}\left( \|\widehat{B}_n\|^2 + \|\hat{\gamma}_n\|^2> M \right) \\
\leq& \mathbb{P}\left( \!\!\|\widehat{B}_n\|^2 \!\!+\!\! \|\hat{\gamma}_n\|^2\!\!\leq\!  \!M \!\!\right) \mathbb{P}\left( \!\!\sup_{(U,\beta)\in K_0} \!\!\!\!|Z_n(U,\beta) \!-\! Z(U,\beta) \!-\! \sigma^2|\!\geq\! \frac{\epsilon}{2} \middle\vert \|\widehat{B}_n\|^2 \!\!+ \!\!\|\hat{\gamma}_n\|^2\!\leq \! M  \!\!\right) \\
&+ \mathbb{P}\left( \|\widehat{B}_n\|^2 + \|\hat{\gamma}_n\|^2> M \right),\\
=& \mathbb{P}\left( \sup_{(U,\beta)\in K_0} |Z_n(U,\beta) - Z(U,\beta) - \sigma^2|\geq \epsilon/2 \right) + \mathbb{P}\left( \|\widehat{B}_n\|^2 + \|\hat{\gamma}_n\|^2> M \right),
\end{align*}
where the last inequality holds since $K_n \subset K_0$ given that $\|\widehat{B}_n\|^2 + \|\hat{\gamma}_n\|^2\leq M$. Thus,  with \eqref{eq: Z_converge_on_k} and \eqref{eq: bound2}, we have
$
\lim_{n\rightarrow \infty} \mathbb{P}\left( |Z(\widehat{B}_n,\hat{\gamma}_n) - Z(\widehat{B},\hat{\gamma})|\geq \epsilon\right) =0,
$
which means that $Z(\widehat{B}_n,\hat{\gamma}_n) \rightarrow Z(\widehat{B},\hat{\gamma})$ in probability as $n\rightarrow \infty$.
The function $Z(\cdot,\cdot)$ is strongly convex under the assumption that $S\succ 0$, which implies that $(\widehat{B}_n,\hat{\gamma}_n)\rightarrow (\widehat{B},\hat{\gamma})$ in probability as $n\rightarrow \infty$. This completes the proof.
\hfill $\blacksquare$

\section{Preliminaries on Non-smooth Analysis}\label{appendix: sec2}
In this section, we give a brief introduction to some basic concepts and some properties, which is critical for developing the semismooth Newton (SSN) method for solving the subsubproblems \eqref{eq: sub_ssn} in {the} preconditioned proximal point algorithm (PPA).

Let $\mathcal{F}:\,\mathbb{R}^n\to \mathbb{R}^m$ be a locally Lipschitz continuous function. Then by Rademacher's theorem, $\mathcal{F}$ is (Fr{\'e}chet) differentiable almost everywhere. Let $D_{\mathcal{F}}$ be the set of points in $\mathbb{R}^n$ where $ \mathcal{F}$ is differentiable and $\mathcal{F}'(x)$ be the Jacobian of $\mathcal{F}$ at $x\in D_{\mathcal{F}}$. The Bouligand subdifferential (B-subdifferential) of $\mathcal{F}$ at any $x\in\mathbb{R}^n$ is defined as
\[ 
\partial_B \mathcal{F}(x)=\left\{\underset{x^k\to x}{\lim}\,\mathcal{F}'(x^k)\,|\,x^k \in D_{\mathcal{F}} \right\},
\]
and the Clarke generalized Jacobian  of $\mathcal{F}$ at  $x\in\mathbb{R}^n$ is defined as {the convex hull of $\partial_B \mathcal{F}(x)$:}
$\partial \mathcal{F}(x)={\rm conv}\{\partial_B \mathcal{F}(x)\}$.
The following definition of  ``semismoothness with respect to a multifunction'', which is mainly adopted from
\citet{qi1993nonsmooth}, \citet{sun2002semismooth}, \citet{mifflin1977semismooth}, \citet{kummer1988newton}, \citet{li2018fused},  is important in  the SSN method.
\begin{definition}
Let $\mathcal{O}\subseteq\mathbb{R}^n$ be an open set, $\mathcal{K}:\mathcal{O} \rightrightarrows \mathbb{R}^{m\times n}$ be a nonempty, compact valued,  upper semicontinuous multifunction, and $\mathcal{F}:\,\mathcal{O} \rightarrow \mathbb{R}^m$ be a locally Lipschitz continuous function. $\mathcal{F}$ is said to be semismooth at $x\in \mathcal{O}$ with respect to the multifunction $\mathcal{K}$ if $\mathcal{F}$ is directionally differentiable at $x$ and for any $V\in\mathcal{K}(x+\Delta x)$ with $\Delta x\rightarrow 0$,
$$
\mathcal{F}(x+\Delta x)-\mathcal{F}(x)-V \Delta x= o(\|\Delta x\|).
$$
Let $\alpha$ be a positive constant. $\mathcal{F}$ is said to be $\alpha$-order (strongly, if $\alpha=1$) semismooth at $x\in\mathcal{O}$ with
respect to $\mathcal{K}$ if $\mathcal{F}$ is directionally differentiable at $x$ and for any $V\in\mathcal{K}(x+\Delta x)$ with $\Delta x\rightarrow 0$,
$$
\mathcal{F}(x+\Delta x)-\mathcal{F}(x)-V \Delta x= O(\|\Delta x\|^{1+\alpha}).
$$
$\mathcal{F}$ is said to be a semismooth (respectively, $\alpha$-order semismooth, strongly semismooth) function on $\mathcal{O}$ with respect to $\mathcal{K}$ if it is semismooth (respectively, $\alpha$-order semismooth, strongly semismooth) everywhere in $\mathcal{O}$  with respect to $\mathcal{K}$.
\end{definition}

\section{Implementation Details}
\subsection{ADMM for Solving the Dual Problem}\label{sec:admm}
In this subsection, we briefly introduce the well known alternating direction method of multipliers (ADMM) for solving \eqref{eq: p}.
With the auxiliary variable {$s\in \mathbb{R}^n$}, the primal problem \eqref{eq: p} is equivalent to
\begin{align} \label{eq: p2}
\min_{B\in \mathbb{R}^{m\times q}, \gamma \in\mathbb{R}^p, s\in \mathbb{R}^n}\ \Big\{
h(s) + \phi(B) + \psi(\gamma) \ \Big| \ X {\rm vec}(B) + Z\gamma -s = 0\Big\}.
\end{align}
It is not difficult to derive the Lagrangian dual problem of \eqref{eq: p2}, given by
\begin{align*}
\min_{\xi\in \mathbb{R}^n} \ \Big\{ h^*(\xi) + \phi^*(-{\rm mat}(X^{\intercal} \xi)) + \psi^*(-Z^{\intercal}\xi)\Big\}.
\end{align*}
And this dual problem can be further written as
\begin{align} \label{eq: d}
\min_{w,\xi\in \mathbb{R}^n,U\in \mathbb{R}^{m\times q},\beta\in \mathbb{R}^p} \! \!\!\Big\{\!
h^*(w) \!+ \! \phi^*(U) \! + \! q^*(\beta) \Big| {\rm mat}(\!X^{\intercal} \xi\!) \!+\! U \!\!=\!\! 0, Z^{\intercal}\xi\!+\!\beta\!\!=\!\!0, w\!-\!\xi\!\!=\!\!0
\!\Big\}.
\end{align}
For $\sigma>0$, the augmented Lagrangian function associated with the dual problem \eqref{eq: d} is
\begin{align*}
&L_{\sigma}(w,\xi,U,\beta;B,\gamma,s) \\
=\;& h^*(w) + \phi^*(U) + \psi^*(\beta)-\langle B,{\rm mat}(X^{\intercal} \xi) + U\rangle - \langle \gamma, Z^{\intercal}\xi+\beta\rangle - \langle s,w-\xi\rangle \\
& + \frac{\sigma}{2}\|{\rm mat}(X^{\intercal} \xi) + U\|^2 +\frac{\sigma}{2}\|Z^{\intercal}\xi+\beta\|^2 + \frac{\sigma}{2}\|w-\xi\|^2\\
=\;&  h^*(w) + \phi^*(U) + \psi^*(\beta)  + \frac{\sigma}{2}\|{\rm mat}(X^{\intercal} \xi) + U -\frac{1}{\sigma} B\|^2 +\frac{\sigma}{2}\|Z^{\intercal}\xi+\beta-\frac{1}{\sigma} \gamma\|^2 \\
& + \frac{\sigma}{2}\|w-\xi-\frac{1}{\sigma} s\|^2 - \frac{1}{2\sigma}\|B\|^2- \frac{1}{2\sigma}\|\gamma\|^2- \frac{1}{2\sigma}\|s\|^2.
\end{align*}
And the ADMM for solving \eqref{eq: d} has the iterations
\begin{align}
& \xi^{k+1} = \underset{\xi }{\arg\min} \ L_{\sigma}(w^k,\xi,U^k,\beta^k;B^k,\gamma^k,s^k), \label{admm-sub1}\\
& (w^{k+1},U^{k+1},\beta^{k+1}) =\underset{w,U,\beta }{\arg\min} \ L_{\sigma}(w,\xi^{k+1},U,\beta;B^k,\gamma^k,s^k), \label{admm-sub2}\\
& B^{k+1} = B^k -\tau\sigma ({\rm mat}(X^{\intercal} \xi^{k+1}) + U^{k+1}), \nonumber\\
& \gamma^{k+1} = \gamma^k  -\tau\sigma(Z^{\intercal}\xi^{k+1}+\beta^{k+1}), \nonumber\\
& s^{k+1} = s^k -\tau\sigma(w^{k+1}-\xi^{k+1}),\nonumber
\end{align}
where ${\tau}\in(0,(1+\sqrt{5})/2)$ is the step size.
Note that the  convergence result of ADMM is well documented
\citep{glowinski1975approximation,Gabay1976dual,Glowinski1984numerical}.

By setting the gradient with respect to $\xi$ to be zero, we can deduce that the subproblem \eqref{admm-sub1} is equivalent to solving a linear system
$
\xi^{k+1} = (I_n + X  X^{\intercal} + Z Z^{\intercal})^{-1} ( X {\rm vec} (\frac{1}{\sigma}B^k-U^k) + Z(\frac{1}{\sigma}\gamma^k-\beta^k) + w^k-\frac{1}{\sigma}s^k).
$
With the definition of the proximal mapping,  problem \eqref{admm-sub2} can be updated as follows:
\begin{align*}
& w^{k+1}  = \xi^{k+1}+\frac{1}{\sigma}s^k - \frac{1}{\sigma}{\rm Prox}_{\sigma h} (\sigma \xi^{k+1} +s^k),\\
& U^{k+1}  =  \frac{1}{\sigma}B^k- {\rm mat}(X^{\intercal}\xi^{k+1}) - \frac{1}{\sigma}{\rm Prox}_{\sigma \phi}(B^k-\sigma {\rm mat}(X^{\intercal}\xi^{k+1})),\\
& \beta^{k+1} =  \frac{1}{\sigma} \gamma^k-Z^{\intercal} \xi^{k+1}-\frac{1}{\sigma} {\rm Prox}_{\sigma \psi}(\gamma^k - \sigma Z^{\intercal} \xi^{k+1}).
\end{align*}

The KKT conditions  associated with \eqref{eq: p2} and \eqref{eq: d} are
\begin{align*}
& X {\rm vec}(B) + Z\gamma -s = 0,  \quad
 {\rm mat}(X^{\intercal} \xi) + U = 0,
\quad Z^{\intercal}\xi+\beta=0,
\quad w-\xi=0,  \\
&  B = {\rm Prox}_{\phi}(B+U),
\quad \gamma = {\rm Prox}_{\psi}(\gamma+\beta),
\quad  s = {\rm Prox}_h(s+w).
\end{align*}
Corresponding to the above KKT conditions, we define the KKT residual terms $R_p,R_d,R_c$:
\begin{equation}\label{kkt-res}
\begin{array}{l}
 R_p :=\frac{\|X {\rm vec}(B) + Z\gamma -s\|}{1+\|s\|},\;\; R_d := \max \left\{\frac{\| {\rm mat}(X^{\intercal}\xi) + U\|_F}{1+\|U\|_F}, \frac{\|Z^{\intercal}\xi +\beta\|}{1+\|\beta\|}, \frac{\|w - \xi\|}{ 1+ \|\xi\|} \right\},\\
 R_c := \max \left\{\frac{\|{\rm Prox}_{\phi}(B+U)-B\|_F}{1+\|B\|_F},\frac{\|{\rm Prox}_{\psi}(\gamma+\beta)-\gamma\|}{1+\|\gamma\|}, \frac{\|{\rm Prox}_h(s+w) - s\|}{ 1+ \|s\|} \right\},
\end{array}
\end{equation}
which represent the primal, dual infeasibilities and complementary condition, respectively.

{
\subsection{Discussions on parameter tuning strategies for $\lambda$ and $\lambda'$}
\label{appendix: lambda}
Setting $\lambda=\lambda'$ reflects a trade-off between the running time of model selection and prediction accuracy. Using distinct values for $\lambda$ and $\lambda'$ would potentially improve the estimation and prediction performance but would also increase the computational cost of model selection.

To investigate this further, we have conducted additional experiments to illustrate the impact of assigning different values to $\lambda$ and $\lambda'$ in model NFL. We follow the experimental setup outlined in Section~\ref{sec:2d}, where the true matrix coefficient exhibits a two-dimensional geometric shape, and the true vector coefficient follows a local constant structure, generated according to scheme (S2). The tuning parameters are set as:
\[\setlength{\abovedisplayskip}{2pt}
  \setlength{\belowdisplayskip}{2pt}
    \rho=\alpha_1\left\|{\rm mat}(X^{\intercal} y)\right\|_2,
    \qquad \lambda =  \alpha_2\left\|Z^{\intercal} y\right\|_{\infty},\qquad  \lambda' = \alpha'_2\left\|Z^{\intercal} y\right\|_{\infty},
\]
where each $\alpha_1,\alpha_2,\alpha'_2$ is selected from a large grid of values in the range of $10^{-3}$ to $1$ with 20 equally divided grid points on the {$\log_{10}$} scale.

The running time for model selection, the estimation error, and the prediction error are presented in Table \ref{Table_two_strategy},  comparing the performance of selecting $\lambda$ and $\lambda'$ independently with that of fixing $\lambda=\lambda'$. We can see from Table \ref{Table_two_strategy} that both estimation error and prediction accuracy have a clear improvement when $\lambda$ and $\lambda'$ are chosen independently. This comes, as expected, at the cost of significantly increased computational time for model selection.

\begin{table}[htbp]
\centering
\caption{Comparison of tuning parameter strategies on model NFL.  ``Fixed $\lambda = \lambda'$'' refers to the tuning parameter strategy where we fix $\lambda=\lambda'$, ``Indep. $\lambda \& \lambda'$'' represents the strategy where we select $\lambda$ and $\lambda'$ independently, and ``Optimal $(\rho,\lambda,\lambda')$'' denotes the parameters identified as optimal through the model selection process.}
\label{Table_two_strategy}
\begin{tabular}{clccccc}
\hline
\textbf{$B$} & Strategy & Optimal $(\rho,\lambda,\lambda')$ & RMSE-$y$ & Error-$B$ & Error-$\gamma$ & Time \\
\hline
\multirow{2}{*}{Square}
& Fixed $\lambda = \lambda'$ & (57.9, 3.04, 3.04) & $9.28$ & $0.11$ & $0.18$ & 0:08:33 \\
& Indep. $\lambda \& \lambda'$ & (57.9, 1.02, 9.03) & $6.44$ & $0.09$ & $0.10$ & 3:21:40 \\
\hline
\multirow{2}{*}{Triangle}
& Fixed $\lambda = \lambda'$ & (585, 35.1, 35.1) & $10.98$ & $0.13$ & $0.23$ & 0:07:38 \\
& Indep. $\lambda \& \lambda'$ & (283, 3.96, 50.5) & $9.12$ & $0.12$ & $0.16$ & 3:06:40 \\
\hline
\multirow{2}{*}{Circle}
& Fixed $\lambda = \lambda'$ & (244, 11.5, 11.5) & $9.82$ & $0.12$ & $0.21$ & 0:08:42 \\
& Indep. $\lambda \& \lambda'$ & (118, 1.86, 16.5) & $7.71$ & $0.10$ & $0.12$ & 3:28:20 \\
\hline
\multirow{2}{*}{Heart}
& Fixed $\lambda = \lambda'$ & (126, 5.57, 5.57) & $10.83$ & $0.13$ & $0.21$ & 0:08:28 \\
& Indep. $\lambda \& \lambda'$ & (541, 11.5, 70.9) & $8.94$ & $0.12$ & $0.15$ & 3:20:00 \\
\hline
\end{tabular}
\end{table}
}

In practice, we suggest considering this trade-off based on the specific application requirements. When computational resources are limited, setting $\lambda=\lambda'$ might offer a practical compromise. However, for  prediction accuracy is important and computational time is enough, tuning $\lambda$ and $\lambda'$ independently can provide better results.

\subsection{Two Dimensional Shapes}\label{sec:2d_appendix}
The true shapes of
$B\in\mathbb{R}^{64\times64}$ used in Section~\ref{sec:2d} are visualized in Figure \ref{fig:images}.

\begin{figure}[H]
\caption{The true shapes of
$B\in\mathbb{R}^{64\times64}$ used in Section~\ref{sec:2d}}
\centering
    \includegraphics[width=0.21\linewidth]{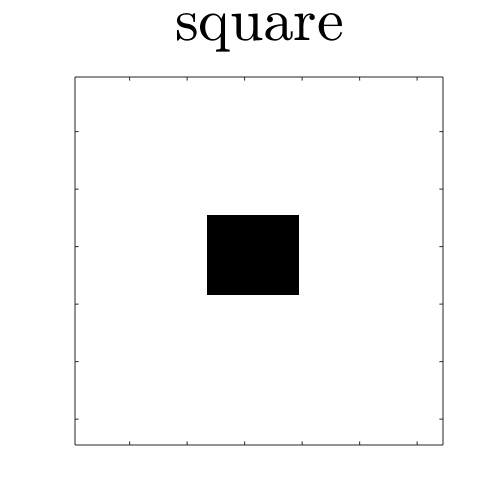}
    \includegraphics[width=0.21\linewidth]{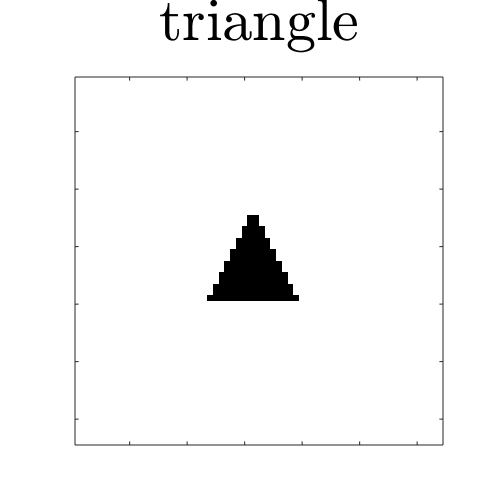}
    \includegraphics[width=0.21\linewidth]{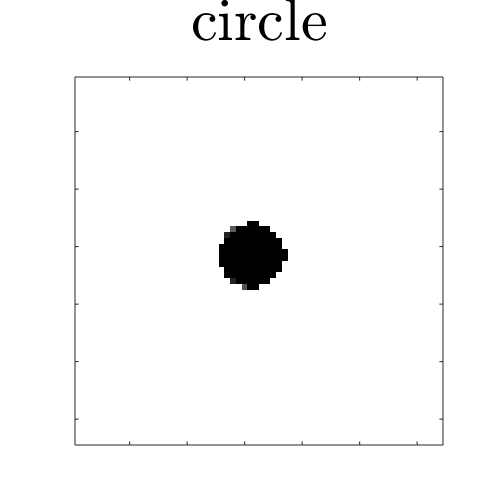}
    \includegraphics[width=0.21\linewidth]{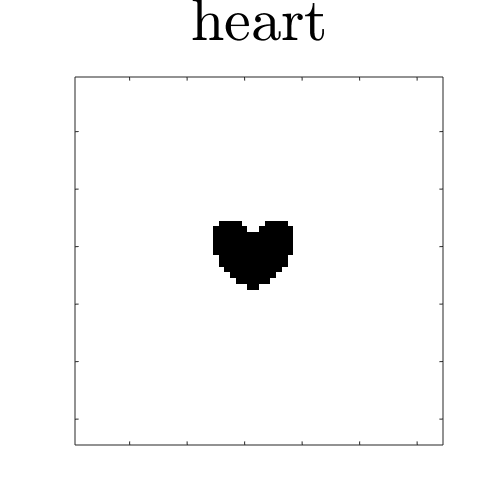}
\label{fig:images}
\end{figure}

\section*{Funding}
    The research of Meixia Lin was supported by the Ministry of Education, Singapore, under its Academic Research Fund Tier 2 grant call (MOE-T2EP20123-0013) and the Singapore University of Technology and Design under MOE Tier 1 Grant SKI 2021\_02\_08. The research of Yangjing Zhang was supported by the National Natural Science Foundation of China under grant number 12201617.

\bibliographystyle{abbrvnat}
\bibliography{references}     

\end{document}